\documentclass[11pt,letterpaper]{amsart} 
\usepackage[margin=0.9in]{geometry}
\usepackage[rgb]{xcolor}
\usepackage{tikz}
\usepackage[bottom]{footmisc}
\usepackage{nicematrix}
\usepackage{verbatim}
\usepackage{mathtools}
\usepackage{amsmath}
\usepackage{amsxtra}
\usepackage{amscd}
\usepackage{amsthm}
\usepackage{amsfonts}
\usepackage{amssymb}
\usepackage{eucal}
\usepackage{float}
\usepackage{mathrsfs}
\usepackage{graphicx}
\usepackage{stmaryrd}
\usepackage{tikz-cd}
\usepackage{mathdots}
\usetikzlibrary{arrows}
\usetikzlibrary{cd}
\usepackage{ytableau}
\usepackage{tikz-cd}
\usetikzlibrary{arrows}
\usetikzlibrary{cd}
\usepackage{caption}
\tikzcdset{every label/.append style = {font = \small}}
\allowdisplaybreaks

\usepackage{latexsym,todonotes}
\PassOptionsToPackage{hyphens}{url}
\usepackage[
    linktoc=all,          
    hyperfootnotes=true,
    colorlinks=true,
    linkcolor=blue,
    urlcolor=orange,
    citecolor=blue
]{hyperref}
\usepackage[all, cmtip]{xy}

\usepackage{stmaryrd}
\usepackage{color} 

\usepackage{amsrefs}

\newtheorem{thm}{Theorem} [section]
\newtheorem{cor}[thm]{Corollary}
\newtheorem{lem}[thm]{Lemma}
\newtheorem{prop}[thm]{Proposition}
\newtheorem{conj}[thm]{Conjecture}

\newtheorem{remark}[thm]{Remark}

\theoremstyle{definition}

\theoremstyle{remark}

\numberwithin{equation}{section}
\begin{document}

\newcommand{\thmref}[1]{Theorem~\ref{#1}}
\newcommand{\secref}[1]{Section~\ref{#1}}
\newcommand{\lemref}[1]{Lemma~\ref{#1}}
\newcommand{\propref}[1]{Proposition~\ref{#1}}
\newcommand{\corref}[1]{Corollary~\ref{#1}}
\newcommand{\remref}[1]{Remark~\ref{#1}}
\newcommand{\eqnref}[1]{(\ref{#1})}
\newcommand{\exref}[1]{Example~\ref{#1}}

 \newcommand{\calA}{\mathcal{A}}
  \newcommand{\calD}{\mathcal{D}}
 \newcommand{\fraksl}{\mathrm{\mathfrak{sl}}}
 \def\calUP{{\mathcal {UP}}}
 \newcommand{\GSp}{\mathrm{GSp}}
 \newcommand{\bbP}{\mathbb{P}}
 \newcommand{\PGSp}{\mathrm{PGSp}}
\newcommand{\PGSO}{\mathrm{PGSO}}
\newcommand{\PGO}{\mathrm{PGO}}
\newcommand{\SO}{\mathrm{SO}}
\newcommand{\GO}{\mathrm{GO}}
\newcommand{\GSO}{\mathrm{GSO}}
\newcommand{\Spin}{\mathrm{Spin}}
\newcommand{\GSpin}{\mathrm{GSpin}}
\newcommand{\Sp}{\mathrm{Sp}}
\newcommand{\PGL}{\mathrm{PGL}}
\newcommand{\GL}{\mathrm{GL}}
\newcommand{\SL}{\mathrm{SL}}
\newcommand{\U}{\mathrm{U}}
\newcommand{\ind}{\mathrm{ind}}
\newcommand{\Ind}{\mathrm{Ind}}
\newcommand{\im}{\mathrm{im}}
\renewcommand{\ker}{\mathrm{ker}}
 \newcommand{\triv}{\mathrm{triv}}
  \newcommand{\std}{\mathrm{std}}
 \newcommand{\Ad}{\mathrm{Ad}}
  \newcommand{\ad}{\mathrm{ad}}
  \newcommand{\Tr}{\mathrm{Tr}}
  \renewcommand{\S}{\mathscr{S}}
  \newcommand{\Y}{\mathbb{Y}}
\newcommand{\End}{\mathrm{End}}
\newcommand{\Ext}{\mathrm{Ext}}
\newcommand{\Mat}{\mathrm{Mat}}
\newcommand{\vol}{\mathrm{vol}}
\newcommand{\bigzero}{\mbox{\normalfont\Large\bfseries 0}}

\newtheorem{innercustomthm}{{\bf Theorem}}
\newenvironment{customthm}[1]
  {\renewcommand\theinnercustomthm{#1}\innercustomthm}
  {\endinnercustomthm}
  
  \newtheorem{innercustomcor}{{\bf Corollary}}
\newenvironment{customcor}[1]
  {\renewcommand\theinnercustomcor{#1}\innercustomcor}
  {\endinnercustomthm}
  
  \newtheorem{innercustomprop}{{\bf Proposition}}
\newenvironment{customprop}[1]
  {\renewcommand\theinnercustomprop{#1}\innercustomprop}
  {\endinnercustomthm}

\newcommand{\bbinom}[2]{\begin{bmatrix}#1 \\ #2\end{bmatrix}}
\newcommand{\cbinom}[2]{\set{\^!\^!\^!\begin{array}{c} #1 \\ #2\end{array}\^!\^!\^!}}
\newcommand{\abinom}[2]{\ang{\^!\^!\^!\begin{array}{c} #1 \\ #2\end{array}\^!\^!\^!}}
\newcommand{\qfact}[1]{[#1]^^!}

\newcommand{\nc}{\newcommand}
\def\C{{\mathbb C}}
\def\R{{\mathbb R}}
\def\Z{{\mathbb Z}}
\def\Q{{\mathbb Q}}
\def\A{{\mathbb A}}
\def\G{{\mathbb G}}
\def\g{{\mathfrak g}}
\def\m{{\mathfrak m}}
\def\k{{\mathfrak k}}
\def\calP{{\mathcal P}}
\def\calI{{\mathcal I}}
\def\calA{{\mathcal A}}
\def\calB{{\mathcal B}}
\def\calF{{\mathcal F}}
\def\calM{{\mathcal M}}
\def\calN{{\mathcal N}}
\def\calT{{\mathcal T}}
\def\calH{{\mathcal H}}
\def\calR{{\mathcal R}}
\def\calFJ{{\mathcal {FJ}}}
\def\calS{{\mathcal S}}
\def\calW{{\mathcal W}}
\def\p{{\mathfrak p}}
\def\sl{{\mathfrak sl}}
\def\su{{\mathfrak su}}
\def\h{{\mathfrak h}}
\def\O{{\mathbb O}}
\def\OO{{\mathcal O}}
\def\tm{{\times}}
\def\sm{{\setminus}}
\def\oomm{{\overline{\omega}}}
\def\D{{\delta}}
\def\Om{{\Omega}}
\newcommand{\e}{{\epsilon}}

\newcommand{\calZ}{\mathcal{Z}}
\newcommand{\omm}{{\omega}}
\newcommand{\RR}{\right}
\newcommand{\LL}{\left}
\newcommand{\floor}[1]{\lfloor #1 \rfloor}
\newcommand{\pair}[1]{\langle {#1} \rangle}

\newcommand{\frakg}{\mathfrak{g}}
\newcommand{\frakgl}{\mathfrak{gl}}
\newcommand{\frakso}{\mathfrak{so}}
\newcommand{\fraksp}{\mathfrak{sp}}
\newcommand{\frako}{\mathfrak{o}}
\newcommand{\fraka}{\mathfrak{a}}
\newcommand{\fraku}{\mathfrak{u}}
\newcommand{\frakp}{\mathfrak{p}}
\newcommand{\fraksu}{\mathfrak{su}}
\newcommand{\frakh}{\mathfrak{h}}
\newcommand{\V}{{\Vert}}
\newcommand{\ord}{\rm {ord}}
\newcommand{\Stab}{Stab}
\newcommand{\Sym}{Sym}
\newcommand{\tr}{\mathrm{tr}}
\newcommand{\Lie}{\mathrm{Lie}}
\newcommand{\rank}{\mathrm{rank}}

\newcommand{\innerproduct}[2]{\langle #1, #2 \rangle}
\newcommand{\HSpin}{\mathrm{HSpin}}
\newcommand{\pr}{\mathrm{pr}}
\newcommand{\GU}{GU}
\newcommand{\SU}{SU}
\newcommand{\Orth}{O}
\newcommand{\glue}{glue}
\newcommand{\diag}{\operatorname{diag}}
\newcommand{\charf}{char}

\newcommand{\ff}{B}

\nc{\etab}{\eta^{\bullet}}
\newcommand{\Iblack}{\I_{\bullet}}
\newcommand{\wb}{w_\bullet}
\newcommand{\UIblack}{\U_{\Iblack}}

\newcommand{\blue}[1]{{\color{blue}#1}}
\newcommand{\red}[1]{{\color{red}#1}}
\newcommand{\green}[1]{{\color{green}#1}}
\newcommand{\white}[1]{{\color{white}#1}}

\newcommand{\dvd}[1]{t_{\odd}^{{(#1)}}}
\newcommand{\dvp}[1]{t_{\ev}^{{(#1)}}}
\newcommand{\ev}{\mathrm{ev}}
\newcommand{\odd}{\mathrm{odd}}

\newcommand\TikCircle[1][2.5]{{\mathop{\tikz[baseline=-#1]{\draw[thick](0,0)circle[radius=#1mm];}}}}

\newcommand{\commentcustom}[1]{}

\raggedbottom

\title[Ginzburg's Conjecture on the Unramified Computation of Eulerian Integrals]
{Ginzburg's Conjecture on the Unramified
Computation of Eulerian Integrals}

\author{Colin Jia Sheng Loh}
 \address{Department of Mathematics, National University of Singapore, 10 Lower Kent Ridge Road, Singapore 119076}
\email{colinloh@u.nus.edu}

 \begin{abstract}
In this paper, we prove a conjecture of D. Ginzburg on the unramified computation of an Eulerian global integral associated with a cuspidal automorphic representation of $\GL_{2r}$ and the generalised Speh representation $\Delta(\tau_2, r)$ attached to a cuspidal automorphic representation $\tau_2$ of $\GL_2$. We also study an analogous Eulerian global integral involving degenerate Eisenstein series and compute its unramified local factors.
 \end{abstract}

\maketitle

\setcounter{tocdepth}{2}
\tableofcontents
\section{Introduction}
In the foundational paper \cite{G}, D. Ginzburg proposed a dimension equation intended to systematically classify Eulerian global integrals. This dimension equation was motivated by common features of previously studied Eulerian global integrals, and has subsequently guided several classifications of Eulerian global integrals. See for instance \cite{G14}, \cite{G16} and \cite{G17}.

As an application, D. Ginzburg introduced an Eulerian global integral defined on $\GL_{2r}$ and conjectured \cite[Conjecture 3]{G} the precise local $L$-factors represented by their unramified local components. This conjecture has been verified for the case of $r=2$ \cite[Proposition 2]{G}. In this paper, we will prove the conjectural unramified identity for all $r\geq2$ and analyse a related Eulerian global integral involving two degenerate Eisenstein series.

\subsection{Notation and preliminaries}
Let $k$ be a number field and $\A$ be its ring of adeles. Let $\psi: k \sm \A \to \C^\tm$ be a nontrivial additive character. For an unipotent group $U$, we write
\begin{align*}
    [U] = U(k) \sm U(\A).
\end{align*}
We equip $k\sm \A$ with the quotient measure normalised such that $\operatorname{vol}(k\sm \A, dx) =1$.
Let $r \geq2$. Let $\pi_{2r}$ and $\tau_2$ be irreducible, cuspidal, globally generic representations of $\GL_{2r}(\A)$ and $\GL_2(\A)$ respectively. We will assume that their central characters denoted by $\omega_{\pi_{2r}}$ and $\omega_{\tau_2}$ are trivial, i.e. $\omega_{\pi_{2r}} = \omega_{\tau_2} = \mathbf 1$.
Let $M_{a,b}$ be the space of $a\tm b$ matrices and $(UT)_r \subset M_{r,r}$ be the subspace of upper triangular matrices. We denote $E_{i,j}$ to be the usual elementary matrix and let $B_n = T_n N_n$ be the standard Borel subgroup of $\GL_n$ consisting of upper triangular matrices with $T_n$ being the diagonal torus of $\GL_n$, and $N_n$ being its unipotent radical, consisting of upper triangular unipotent matrices. We set $\psi_{N_n} : N_n(k) \sm N_n(\A) \to \C^\tm$ to be the character defined by
\begin{align*}
    \psi_{N_n}(u) = \psi\LL( \sum_{i=1}^{n-1} u_{i,i+1} \RR), && u = (u_{i,j}) \in N_n(\A).
\end{align*}
With this, we define the Whittaker coefficient $W_{\varphi_{\pi_{2r}}}^\psi$ of $\varphi_{\pi_{2r}} \in \pi_{2r}$ as
\begin{align}
    \label{eqn: Whittaker coefficient pi-2r}
    W_{\varphi_{\pi_{2r}}}^\psi(g) = \int_{[N_{2r}]} \varphi_{\pi_{2r}}(ng) \psi_{N_{2r}}^{-1}(n) \, dn, && g \in \GL_{2r}(\A).
\end{align}
For $1 \leq i \leq n-1$, we let $U_i = U_{i,n} \subset N_n$ be the abelian unipotent subgroup given by
\begin{align}
\label{eqn: U-i}
    U_{i} = \LL\{\begin{pmatrix}
I_i & X &  \\
 & 1 &  \\
 &  & I_{n-i-1}
\end{pmatrix}
\ \mid X \in M_{i,1} \RR\}.
\end{align}
Let $\Delta(\tau_2, r)$ be the generalised Speh representation of $\GL_{2r}(\A)$ associated to $\tau_2$ (see \cite{CFGK} for its definition). If $(n_1,\dots, n_t)$ is a composition of $N$, let $P_{n_1,\dots, n_t}$ be the standard parabolic subgroup of $\GL_N$ with Levi subgroup $M_{n_1,\dots, n_t} \cong \GL_{n_1} \tm \cdots \tm \GL_{n_t}$.

For $s, s_1,s_2 \in \C$, define $I(s_1,s_2)$ as the induced representation defined by
\begin{align*}
    I(s_1,s_2) = \Ind_{P_{r,r}(\A)}^{\GL_{2r}(\A)}\LL( \mathbf 1 \otimes
        \Ind_{P_{r-1,1}(\A)}^{\GL_r(\A)} \D_{P_{r-1,1}}^{s_2}
    \RR) \D_{P_{r,r}}^{s_1} \cong 
    \Ind_{P_{r,r-1,1}(\A)}^{\GL_{2r}(\A)} \D_{P_{r,r}}^{s_1} \D_{P_{r-1,1}}^{s_2},
\end{align*} 
where $I(s_1,s_2)$ consists of complex-valued functions $f(\cdot,s_1,s_2) : \GL_{2r}(\A) \to \C$ with the following quasi-invariance:
\begin{align}
    \label{eqn: f(s1,s2) quasi-invariance}
    f(\begin{pmatrix}
        g_1 &\ast&\ast\\
        &g_2&\ast\\
        &&t
    \end{pmatrix}h,s_1,s_2) = 
    |\det g_1|^{rs_1} |\det g_2|^{-rs_1+s_2}
    |t|^{-rs_1-(r-1)s_2} f(h,s_1,s_2),
\end{align}
for $g_1 \in \GL_r(\A), g_2 \in \GL_{r-1}(\A)$, $t \in \GL_1(\A)$ and $h \in \GL_{2r}(\A)$. Similarly, we define
\begin{align*}
    I(s) = \Ind_{P_{r,r}(\A)}^{\GL_{2r}(\A)} \D_{P_{r,r}}^s
\end{align*}
consisting of functions $f(\cdot, s) : \GL_{2r}(\A) \to \C$ satisfying the following quasi-invariance:
\begin{align}
    \label{eqn: f(s) quasi-invariance}
    f(\begin{pmatrix}
        h_1 & X\\ & h_2
    \end{pmatrix} h, s) =|\det h_1|^{rs} |\det h_2|^{-rs} f(h, s),
\end{align}
for $h_1, h_2 \in \GL_r(\A)$, $X \in M_{r,r}(\A)$ and $h \in \GL_{2r}(\A)$.
We define their corresponding Eisenstein series $E(\cdot, s_1,s_2)$ and $E(\cdot, s)$ as
\begin{align}
    \label{eqn: Eisenstein series (s_1,s_2)}
    E(g, s_1,s_2) =& \sum_{\gamma \in P_{r,r-1,1}(k) \sm \GL_{2r}(k)} f(\gamma g, s_1,s_2),\\
    \label{eqn: Eisenstein series (s)}
    E(g, s) =& \sum_{\gamma \in P_{r,r}(k) \sm \GL_{2r}(k)} f(\gamma g, s),
\end{align}
for $g \in \GL_{2r}(\A)$, $f(\cdot,s_1,s_2) \in I(s_1,s_2)$ and $f(\cdot,s) \in I(s)$.
For $\varphi_{\pi_{2r}}\in \pi_{2r}$ and $\theta_{\tau_2} \in \Delta(\tau_2,r)$, define the global integrals $I_r(\varphi_{\pi_{2r}}, \theta_{\tau_2}, s_1,s_2)$ and $I_r(\varphi_{\pi_{2r}}, s,s_1,s_2)$ as
\begin{align}
    \label{eqn: global integral tau2}
    I_r(\varphi_{\pi_{2r}}, \theta_{\tau_2}, s_1,s_2)=& \int_{Z_{2r}(\A) \GL_{2r}(k) \sm \GL_{2r}(\A)}
    \varphi_{\pi_{2r}}(g) \theta_{\tau_2}(g)
    E(g,s_1,s_2) \, dg,\\
    \label{eqn: global integral (s,s1,s2)}
    I_r(\varphi_{\pi_{2r}},s,s_1,s_2) =& \int_{Z_{2r}(\A) \GL_{2r}(k) \sm \GL_{2r}(\A)}
    \varphi_{\pi_{2r}}(g) E(g,s) E(g,s_1,s_2)\,dg.
\end{align}
Here $Z_{2r}$ denotes the center of $\GL_{2r}$.
In \cite[Conjecture 3]{G}, D. Ginzburg conjectured the following.
\begin{conj}
    \label{conj: Ginzburg conjecture}
    \cite[Conjecture 3]{G}
    \footnote{There are some typographical inaccuracies in the formulation
of \cite[Conjecture 3]{G}, which we have corrected in the above
Conjecture \ref{conj: Ginzburg conjecture}.}
    The global integral $I_r(\varphi_{\pi_{2r}}, \theta_{\tau_2}, s_1,s_2)$ defined in \eqref{eqn: global integral tau2} is Eulerian and for $\operatorname{Re}(2rs_1 - s_2),\operatorname{Re}(s_1), \operatorname{Re}(s_2) \gg0$ the global integral unfolds to 
    \begin{align}
        \label{eqn: unfolded integral}
        I_r(\varphi_{\pi_{2r}}, \theta_{\tau_2}, s_1,s_2)= \int_{Z_{2r}(\A) N_{2r}(\A) \sm \GL_{2r}(\A)}
        W_{\varphi_{\pi_{2r}}}^{\psi}(g)
        W_{\theta_{\tau_2},2r}(g) 
        f_{W_{2r-1}}(g, s_1,s_2) \, dg.
    \end{align}
    Moreover, for a non-archimedean local field $F$, and unramified local integration data $W_{\pi_{2r}}^\circ, W_{\Delta(\tau_2,r)}^\circ$ and $f_{W_{2r-1}^\circ}$, its unramified local integral $\calZ_r(W^\circ, s_1,s_2)$ defined by
    \begin{align}
        \label{eqn: unramified local integral}
        \calZ_r(W^\circ, s_1,s_2) = \int_{Z_{2r}(F) N_{2r}(F) \sm \GL_{2r}(F)}
        W_{\pi_{2r}}^\circ(g)
        W_{\Delta(\tau_2,r)}^\circ(g)
        f_{W_{2r-1}^\circ}(g,s_1,s_2)\, dg
    \end{align}
    evaluates to 
    \begin{align*}
          \frac{
        L\LL(\pi_{2r} \tm \tau_2, rs_1 - \frac{r-1}{2} \RR) L\LL(\pi_{2r},\wedge^2,s_2\RR) 
        }{
        \displaystyle\zeta_F(rs_2) 
        \zeta_F(2rs_1+(r-1)s_2-r+1)
        \prod_{i=1}^{r-1} \zeta_F(2rs_1 - s_2 -i + 1)
        }.
    \end{align*}
    Here $L(\pi_{2r}\tm \tau_2, s)$ is the local tensor $L$-function and $L(\pi_{2r}, \wedge^2, s)$ is the local exterior square $L$-function of $\pi_{2r}$.
\end{conj}
The above conjecture has been verified for the base case of $r =2$ by D. Ginzburg \cite[Proposition 2]{G}. Moreover, he speculated that the global integral $I_r(\varphi_{\pi_{2r}}, s,s_1,s_2)$ defined in \eqref{eqn: global integral (s,s1,s2)} is an integral representation for the product of $L$-functions:
\begin{align*}
    L(\pi_{2r},s) L(\pi_{2r},s_1) L(\pi_{2r},\wedge^2,s_2).
\end{align*}
\subsection{Main result and organisation of paper}
In this paper, we prove Conjecture \ref{conj: Ginzburg conjecture} and verify Ginzburg's speculation. More specifically, we have the following main theorem. For precise definitions of the functions $W_{(r,r);2r}(\cdot,s)$ and $f_{W_{2r-1}}(\cdot,s_1,s_2)$ etc, see Sections \ref{sec: Unfolding of global integrals} and \ref{sec: Unramified computation}.
\begin{thm}
(Proposition \ref{prop: unfolding} and Theorem \ref{thm: unramified computation})
    \label{thm: main thm}
    \begin{enumerate}
        \item [(a)] Conjecture \ref{conj: Ginzburg conjecture} holds for all $r \geq 2$.
        \item [(b)] Similarly, the global integral $I_r(\varphi_{\pi_{2r}},s,s_1,s_2)$ defined in \eqref{eqn: global integral (s,s1,s2)} is Eulerian, and for $\operatorname{Re}(s_1-s) ,\operatorname{Re}(2rs_1-s_2), \operatorname{Re}(s), \operatorname{Re}(s_1),\operatorname{Re}(s_2) \gg0$ the global integral unfolds to 
        \begin{align}
        \label{eqn: I-r-s,s1,s2}
            I_r(\varphi_{\pi_{2r}}, s,s_1,s_2) = \int_{Z_{2r}(\A) N_{2r}(\A) \sm \GL_{2r}(\A)}
            W_{\varphi_{\pi_{2r}}}^\psi(g) 
            W_{(r,r);2r}(g,s)
            f_{W_{2r-1}}(g,s_1,s_2)\, dg.
        \end{align}
        Moreover, for a non-archimedean local field $F$, and unramified local integration data $W_{\pi_{2r}}^\circ, W_{(r,r);2r}^\circ$ and $f_{W_{2r-1}^\circ}$, its local integral $\calZ_r(W^\circ,s,s_1,s_2)$ defined by
        \begin{align*}
            \calZ_r(W^\circ,s,s_1,s_2)= \int_{Z_{2r}(F) N_{2r}(F) \sm \GL_{2r}(F)}
            W_{\pi_{2r}}^\circ(g)
            W_{(r,r);2r}^\circ(g,s)
            f_{W_{2r-1}^\circ}(g,s_1,s_2)\,dg
        \end{align*}
        evaluates to 
        \begin{align*}
             \frac{
        L\LL(\pi_{2r}, r(s_1+s-1) + \frac{1}{2} \RR)
        L\LL(\pi_{2r}, r(s_1-s) + \frac{1}{2} \RR)
        L\LL(\pi_{2r},\wedge^2,s_2\RR) 
        }{
        \displaystyle \zeta_F(rs_2) 
        \zeta_F(2rs_1+(r-1)s_2-r+1)
        \prod_{i=1}^{r-1} \zeta_F(2rs_1 - s_2 -i + 1)
        \prod_{j=1}^r \zeta_F(2rs - r + j)
        }.
        \end{align*}
    \end{enumerate}
\end{thm}
\begin{remark}
    Let $r = 2$. We note that up to a functional equation of the Eisenstein series $E(\cdot,s_1,s_2)$ defined in \eqref{eqn: Eisenstein series (s_1,s_2)}, the global integral $I_2(\varphi_{\pi_4}, s,s_1,s_2)$ in \eqref{eqn: global integral (s,s1,s2)} was studied in \cite{PS} by Pollack and Shah.
\end{remark}
The paper is organised as follows. In Section \ref{sec: Unfolding of global integrals} we unfold both global integrals and identify the Fourier coefficients in their resulting Euler product decomposition. Next, in Section \ref{sec: Unramified computation} we compute the unramified local integrals via reducing them to certain rank-one computations and the use of a Cauchy-Littlewood type identity.
\subsection{Acknowledgements}
I would like to thank Wee Teck Gan and Lei Zhang for providing advice and helpful discussions. Additionally, I would like to thank David Ginzburg for clarifications and insights on his work \cite{G}.

\section{Unfolding of global integrals}
\label{sec: Unfolding of global integrals}
In this section, we will provide details of the unfolding process of the global integrals to obtain \eqref{eqn: unfolded integral} and \eqref{eqn: I-r-s,s1,s2}. More precisely, we will prove the following proposition.
\begin{prop}
    \label{prop: unfolding}
    The global integrals $I_r(\varphi_{\pi_{2r}}, \theta_{\tau_2}, s_1,s_2)$ and $I_r(\varphi_{\pi_{2r}}, s,s_1,s_2)$  defined in \eqref{eqn: global integral tau2} and \eqref{eqn: global integral (s,s1,s2)} are Eulerian and for $\operatorname{Re}(s_1-s),\operatorname{Re}(2rs_1-s_2),\operatorname{Re}(s),\operatorname{Re}(s_1), \operatorname{Re}(s_2) \gg0$, they unfold to 
    \begin{align*}
        I_r(\varphi_{\pi_{2r}}, \theta_{\tau_2}, s_1,s_2)=&  \int_{Z_{2r}(\A) N_{2r}(\A) \sm \GL_{2r}(\A)} W_{\varphi_{\pi_{2r}}}^{\psi}(g) W_{\theta_{\tau_2},2r}(g)
        f_{W_{2r-1}}(g,s_1,s_2)\,dg,\\
        I_r(\varphi_{\pi_{2r}},s,s_1,s_2)=& \int_{Z_{2r}(\A) N_{2r}(\A) \sm \GL_{2r}(\A)}
        W_{\varphi_{\pi_{2r}}}^{\psi}(g) 
        W_{(r,r);2r}(g,s)
        f_{W_{2r-1}}(g,s_1,s_2)\,dg,
    \end{align*}
    where $W_{\varphi_{\pi_{2r}}}^\psi$ is as defined in \eqref{eqn: Whittaker coefficient pi-2r}, and for $g \in \GL_{2r}(\A)$, $ W_{\theta_{\tau_2},2r}(g)$, $W_{(r,r);2r}(g,s)$ and $f_{W_{2r-1}}(g, s_1,s_2)$ are defined by
\begin{align}
    \label{eqn: W-theta-2r}
    W_{\theta_{\tau_2},2r}(g) =& \int_{N_{2r}(k) \sm N_{2r}(\A)} \theta_{\tau_2}(n g) \psi_{N_{2r}}^{\operatorname{odd}}(n) \, dn,\\
    \label{eqn: W(r,r)(s)}
    W_{(r,r);2r}(g,s) =& \int_{U_{w_0'}(\A)} f(w_0'ug,s)\psi_{U_{w_0'}}(u)\, du,\\
    \label{f-W(2r-1)}
    f_{W_{2r-1}}(g, s_1,s_2) =& \int_{U_{w_0}(\A)} 
    f(w_0 n g, s_1, s_2) \psi_{N_{2r}}^{\operatorname{even}}(n) \, dn.
\end{align}
Here $U_{w_0'}$ and $U_{w_0}$ are subgroups of $N_{2r}$ given by
\begin{align*}
    U_{w_0'} =& \LL\{ u(X) = I_{2r} + \sum_{1 \leq i \leq j \leq r} x_{i,j} E_{2i-1,2j} \mid X = (x_{i,j}) \in (UT)_r  \RR\},\\
    U_{w_0} =& \LL\{
    n(u) = I_{2r} + \sum_{1 \leq i < j \leq r} u_{2i,2j-1} E_{2i,2j-1}
    \RR\},
\end{align*}
and $w_0, w_0' \in \GL_{2r}$ are permutation matrices corresponding to permutations $s_{w_0}$ and $s_{w_0'}$ respectively:
\begin{align*}
    s_{w_0} = \begin{pmatrix}
        1 & 2 & 3 & 4 & \cdots & 2r-1 & 2r\\
        1 & r+1 & 2 & r+2 & \cdots & r & 2r
    \end{pmatrix}, && 
    s_{w_0'} = \begin{pmatrix}
        1 & 2& 3 &4 & \cdots & 2r-1 & 2r\\
        r+1 & 1 & r+2 & 2 & \cdots & 2r & r
    \end{pmatrix}.
\end{align*}
Also, $\psi_{N_{2r}}^{\operatorname{even}}, \psi_{N_{2r}}^{\operatorname{odd}}$ and $\psi_{U_{w_0'}}$ are characters given by
\begin{align*}
    \psi_{N_{2r}}^{\operatorname{odd}}(n) =\psi\LL(\sum_{i=1}^{r} n_{2i-1,2i} \RR), && 
    \psi_{N_{2r}}^{\operatorname{even}}(n) = \psi\LL( \sum_{i=1}^{r-1} n_{2i,2i+1} \RR), && 
    \psi_{U_{w_0'}}(u(X)) = \psi(\tr X),
\end{align*}
for $n = (n_{i,j}) \in N_{2r}$ and $u(X) \in U_{w_0'}$.
\end{prop}
As noted in \cite{G}, the proof of Proposition \ref{prop: unfolding} follows from similar arguments in \cite{BG}. For completeness and the reader's convenience, we will provide details of the proof of the above proposition. As in \cite{BG}, we will first establish some preliminary results. Assuming $\operatorname{Re}(s_1), \operatorname{Re}(s_2) \gg0$, we can consider the following expression for the Eisenstein series $E(\cdot,s_1,s_2)$ defined in \eqref{eqn: Eisenstein series (s_1,s_2)},
\begin{align}
        \label{eqn: Eisenstein series II}
        E(g,s_1,s_2) = \sum_{\gamma_1 \in P_{2r-1,1}(k) \sm \GL_{2r}(k)} F(\gamma_1 g, s_1,s_2), 
    \end{align}
where
\begin{align}
    \label{eqn: F auxiliary function}
        F(g,s_1,s_2) = \sum_{\gamma_2 \in P_{r,r-1,1}(k) \sm P_{2r-1,1}(k)} f(\gamma_2 g, s_1, s_2).
\end{align}
Let $\phi$ be either $\theta_{\tau_2} \in \Delta(\tau_2, r)$ or the Eisenstein series $E(\cdot, s)$ defined in \eqref{eqn: Eisenstein series (s)}. For $1 \leq m \leq r-1$ and $g \in \GL_{2r}(\A)$, we define
\begin{align*}
     A_m^\phi(g) = \int_{[U_{2r-2m+1}]} \cdots \int_{[U_{2r-2}]} &\int_{[U_{2r-1}]} \phi(u_{2r-1}u_{2r-2}\cdots u_{2r-2m+1}g)\\
    \cdot&\psi_{N_{2r}}\LL( u_{2r-1} u_{2r-3}\cdots  u_{2r-2m+1} \RR)\, du_{2r-1}du_{2r-2}\cdots du_{2r-2m+1}.
\end{align*}
Furthermore, for $2 \leq m \leq r$ and $g \in \GL_{2r}(\A)$, we set
\begin{align*}
   B_m(g,s_1,s_2) =\int_{[U_{2r-2m+2}]} &\cdots \int_{[U_{2r-3}]}
    \int_{[U_{2r-2}]} F(
    u_{2r-2}u_{2r-3}\cdots u_{2r-2m+2}
    g,s_1,s_2)\\
    \cdot&\psi_{N_{2r-1}}(u_{2r-2}u_{2r-4}\cdots u_{2r-2m+2})\, du_{2r-2}\, du_{2r-3}\cdots du_{2r-2m+2}.
\end{align*}
\begin{lem}
    \label{lemma: prop2.4 of Bump-Ginzburg}
    Let $g \in \GL_{2r}(\A)$ and let $\phi$ denote either $\theta_{\tau_2}$ or $E(\cdot,s)$.
    \begin{enumerate}
        \item [(a)]
        The function $u \mapsto A_m^\phi(ug)$ is constant on $U_{2r-2m}(\A)$ for all $1\leq m \leq r-1$;
        \item [(b)]
        the function $u \mapsto B_m(u g,s_1,s_2)$ is constant on $U_{2r-2m+1}(\A)$ for all $2 \leq m \leq r$.
    \end{enumerate}
\end{lem}
\begin{proof}
    The proof follows from the same argument as \cite[Proposition 2.3]{BG}. We will only provide a proof of identity $(a)$ as the same argument holds for identity $(b)$. Let $1 \leq m \leq r-1$, we define
    \begin{align*}
        F_m(g) = \int_{[U_{2r-2m+1,2r}]}\cdots \int_{[U_{2r-2,2r}]} \int_{[U_{2r-1,2r}]}
        &\phi(u_{2r-1}u_{2r-2}u_{2r-3}\cdots u_{2r-2m+1}g) \\
        \cdot&
        \psi_{N_{2r}}(u_{2r-1}u_{2r-3}u_{2r-5}\cdots u_{2r-2m+1})\,\, du_{2r-1}\cdots \, du_{2r-2m+1},
    \end{align*}
    and 
    \begin{align*}
        G_m(g) = F_m(g) - \int_{[U_{2r-2m}]} F_m(u_{2r-2m} g) \, du_{2r-2m}.
    \end{align*}
    We will show that $G_m(g) =0$. By suitable change of variables, the function $u_{2r-2m} \mapsto F_m(u_{2r-2m}g)$ is a well-defined function of $U_{2r-2m}(k) \sm U_{2r-2m}(\A)$ and by a standard Fourier expansion argument we have
    \begin{align*}
        F_m(g) = \int_{[U_{2r-2m}]} F_{m}(u_{2r-2m} g) \, du_{2r-2m} 
        + \sum_{\chi} \int_{[U_{2r-2m}]}
        F_{m}(u_{2r-2m} g) \chi(u_{2r-2m})\, du_{2r-2m},
    \end{align*}
    where the sum is over all nontrivial characters $\chi$ on $[U_{2r-2m}]$. Then, we can write $G_m(g)$ as
    \begin{align*}
        G_m(g)
        =& 
        \sum_{\gamma \in P'_{2r-2m}(k) \sm \GL_{2r-2m}(k)}
        \int_{[U_{2r-2m}]}
        F_m(u_{2r-2m} \begin{pmatrix}
            \gamma & \\ & I_{2m}
        \end{pmatrix}g) \psi_{N_{2r}}(u_{2r-2m}) \, du_{2r-2m},
    \end{align*}
    where for $n\geq 2$, $P_n'\subset \GL_n$ is the usual mirabolic subgroup of $\GL_n$ given by
    \begin{align*}
        P_n' = \LL\{ \begin{pmatrix}
            h & v \\ & 1     
        \end{pmatrix} \mid h \in \GL_{n-1}, v \in M_{n-1,1} \RR\}.
    \end{align*}
    Thus, to show that $G_m(g) =0$, it suffices to show that 
    \begin{align*}
        F_m^\ast(\gamma, g) =
         \int_{[U_{2r-2m}]}
         F_m(u_{2r-2m} \begin{pmatrix}
            \gamma & \\ & I_{2m}
        \end{pmatrix}g) \psi_{N_{2r}}(u_{2r-2m}) \, du_{2r-2m}
    \end{align*}
    vanishes for all $\gamma \in \GL_{2r-2m}(k),g \in \GL_{2r}(\A)$. Indeed, from the definition of $F_m$, we have
    \begin{align*}
        F_m^\ast(\gamma, g) = \int_{[U_{2r-2m}]} \int_{[U_{2r-2m+1}]} \cdots &\int_{[U_{2r-2}]} 
        \int_{[U_{2r-1}]}
        \phi(u_{2r-1} u_{2r-2}\cdots 
        u_{2r-2m+1} u_{2r-2m}\begin{pmatrix}
            \gamma & \\ & I_{2m} 
        \end{pmatrix} g)\\
        \cdot&\psi_{N_{2r}}(u_{2r-1}u_{2r-3}\cdots u_{2r-2m+1} u_{2r-2m})\, du_{2r-1}\, du_{2r-2}\cdots du_{2r-2m}.
    \end{align*}
    Moreover, since $\OO(\Delta(\tau_2,r)) = (2^r)$ and $\OO(E(\cdot,s))= (2^r)$ as established in \cite{G06b} (see also \cite{JL}) and \cite{Cai} respectively, we conclude $F_m^\ast(\gamma, g)$ evaluates to zero for all $\gamma \in \GL_{2r-2m}(k)$, $g \in \GL_{2r}(\A)$ and $\phi \in \{\theta_{\tau_2}, E(\cdot,s)\}$. Consequently, we have $G_m(g) =0$ and this concludes the proof.
\end{proof}
As a consequence, we have the following corollary.
\begin{cor}
Let $g \in \GL_{2r}(\A)$ and let $\phi$ denote either $\theta_{\tau_2}$ or $E(\cdot,s)$.
    \label{cor: independence}
    \begin{enumerate}
        \item [(a)] For $1 \leq m \leq r-1$, we have
        \begin{align*}
        A_m^\phi(g) = \int_{[U_{2r-2m}]}\int_{[U_{2r-2m+1}]} &\cdots \int_{[U_{2r-2}]} \int_{[U_{2r-1}]} \phi(u_{2r-1}u_{2r-2}\cdots u_{2r-2m+1} u_{2r-2m} g)\\
    \cdot&\psi_{N_{2r}}\LL( u_{2r-1} u_{2r-3}\cdots  u_{2r-2m+1} \RR)\, du_{2r-1}du_{2r-2}\cdots du_{2r-2m+1}\, du_{2r-2m},
    \end{align*}

        \item [(b)] For $2 \leq m \leq r$, we have
        \begin{align*}
            B_m(g,s_1,s_2)=\int_{[U_{2r-2m+1}]}\int_{[U_{2r-2m+2}]} &\cdots \int_{[U_{2r-3}]}
    \int_{[U_{2r-2}]} F(u_{2r-2} u_{2r-3} \cdots u_{2r-2m+2} u_{2r-2m+1}g,s_1,s_2)\\
    \cdot&\psi_{N_{2r-1}}(u_{2r-2}u_{2r-4}\cdots u_{2r-2m+2})\, du_{2r-2}\cdots du_{2r-2m+1}.
        \end{align*}
    \end{enumerate}
\end{cor}
\begin{proof}
    Since the measures are chosen such that the volume of $[U_{2r-2m}]$ is one for $U_{2r-2m} \subset \GL_{2r}$, $1 \leq m \leq r-1$, and also the volume of $[U_{2r-2m+1}]$ is one for $U_{2r-2m+1} \subset \GL_{2r-1}$ for $2 \leq m \leq r$, the identities above follow from Lemma \ref{lemma: prop2.4 of Bump-Ginzburg}.
\end{proof}
With these, we can proceed with the proof of Proposition \ref{prop: unfolding}. Again, we will let $\phi$ denote either $\theta_{\tau_2}$ or $E(\cdot, s)$ and consider the integral 
\begin{align*}
    I_r(\varphi_{\pi_{2r}}, \phi, s_1, s_2)= \int_{Z_{2r}(\A) \GL_{2r}(k) \sm \GL_{2r}(\A)}
    \varphi_{\pi_{2r}}(g) \phi(g) E(g, s_1, s_2) \, dg,
\end{align*}
such that $I_r(\varphi_{\pi_{2r}}, E(\cdot,s), s_1, s_2) = I_r(\varphi_{\pi_{2r}}, s, s_1, s_2)$ given in \eqref{eqn: global integral (s,s1,s2)}. In the region of absolute convergence, we can unfold the Eisenstein series $E(\cdot, s_1, s_2)$ via \eqref{eqn: Eisenstein series II} to obtain
\begin{align*}
    I_r(\varphi_{\pi_{2r}}, \phi, s_1, s_2)= \int_{Z_{2r}(\A) P_{2r-1,1}(k) \sm \GL_{2r}(\A)}
    \varphi_{\pi_{2r}}(g) \phi(g) F(g, s_1, s_2) \, dg.
\end{align*}
Next, we recall the Shalika-Piatetski-Shapiro expansion for the cusp form $\varphi_{\pi_{2r}}$: 
\begin{align*}
    \varphi_{\pi_{2r}}(g) = \sum_{\gamma \in N_{2r-1}(k) \sm \GL_{2r-1}(k)} W_{\varphi_{\pi_{2r}}}^\psi(\begin{pmatrix}
        \gamma & \\ & 1
    \end{pmatrix} g), && g \in \GL_{2r}(\A).
\end{align*}
Since $P_{2r-1,1} = Z_{2r} U_{2r-1} \GL_{2r-1}$ where $U_{2r-1}$ is as given in \eqref{eqn: U-i} and $\GL_{2r-1}$ is realised as a subgroup of $\GL_{2r}$ via
\begin{align*}
    \GL_{2r-1} \cong \LL\{ \begin{pmatrix}
        g & \\ & 1
    \end{pmatrix} \mid g \in \GL_{2r-1} \RR\},
\end{align*}
we can write the global integral $I_r(\varphi_{\pi_{2r}}, \phi, s_1,s_2)$ as 
\begin{align*}
    I_r(\varphi_{\pi_{2r}}, \phi, s_1,s_2)
    =& \int_{Z_{2r}(\A) N_{2r}(k) \sm \GL_{2r}(\A)}
    W_{\varphi_{\pi_{2r}}}^\psi(g) \phi(g) 
    F(g,s_1,s_2)\, dg.
\end{align*}
Writing $N_{2r} =U_{2r-1}  \cdots U_{2} U_1$, then following the definition of $F(g,s_1,s_2)$ in \eqref{eqn: F auxiliary function}, we have $I_r(\varphi_{\pi_{2r}}, \phi, s_1,s_2) $ equals to 
\begin{align*}
    &\int_{Z_{2r}(\A) U_{2r-1}(\A) \cdots U_{2}(k) U_{1}(k)\sm \GL_{2r}(\A)}
    W_{\varphi_{\pi_{2r}}}^\psi(g) \LL(\int_{[U_{2r-1}]} \phi(u_{2r-1} g) \psi_{N_{2r}}(u_{2r-1})\, du_{2r-1}\RR)
    F(g, s_1,s_2) \, dg.
\end{align*}
Next, factorising the integral over $U_{2r-2}$ we have
\begin{align}
    \label{eqn: I-r integral II}
    I_r(\varphi_{\pi_{2r}},& \phi, s_1,s_2) =
    \int_{Z_{2r}(\A) U_{2r-1}(\A) U_{2r-2}(\A) U_{2r-3}(k) \cdots U_1(k)\sm \GL_{2r}(\A)} W_{\varphi_{\pi_{2r}}}^\psi(g)\\
    \cdot&
    \int_{[U_{2r-2}]} \int_{[U_{2r-1}]}
    \phi(u_{2r-1}u_{2r-2}g) 
    F(u_{2r-2}g,s_1,s_2)
    \psi_{N_{2r}}(u_{2r-1}u_{2r-2})\, du_{2r-1}\,du_{2r-2}\, dg.\notag
\end{align}
By Corollary \ref{cor: independence}, we have
\begin{align*}
    \int_{[U_{2r-1}]} \int_{[U_{2r-2}]} 
    \phi(u_{2r-1}u_{2r-2}g) \psi_{N_{2r}}(u_{2r-1})
    \, du_{2r-1}\,du_{2r-2} =
    \int_{[U_{2r-1}]} \phi(u_{2r-1}u_{2r-2}g) 
    \psi_{N_{2r}}(u_{2r-1})\, du_{2r-1},
\end{align*}
and substituting this identity into \eqref{eqn: I-r integral II}, we obtain
\begin{align*}
     &I_r(\varphi_{\pi_{2r}}, \phi, s_1,s_2) =
     \int_{Z_{2r}(\A) U_{2r-1}(\A) U_{2r-2}(\A) U_{2r-3}(k) \cdots U_1(k)\sm \GL_{2r}(\A)} W_{\varphi_{\pi_{2r}}}^\psi(g) \\
     \cdot&
     \LL( \int_{[U_{2r-2}]} \int_{[U_{2r-1}]} \phi(u_{2r-1}u_{2r-2}g) \psi_{N_{2r}}(u_{2r-1})\, du_{2r-1}\, du_{2r-2} \RR)
     \int_{[U_{2r-2}]} F(u_{2r-2}g,s_1,s_2)\,\psi_{N_{2r}}(u_{2r-2})\, du_{2r-2}.
\end{align*}
Next, proceeding with factoring the integral over $U_{2r-3}$, we have
\begin{align}
    \label{eqn: I-r integral III}
    I_r(&\varphi_{\pi_{2r}}, \phi, s_1,s_2) =
    \int_{Z_{2r}(\A) U_{2r-1}(\A) U_{2r-2}(\A) U_{2r-3}(\A) U_{2r-4}(k) \cdots U_{1}(k) \sm\GL_{2r}(\A)} W_{\varphi_{\pi_{2r}}}^\psi(g) \\
    \cdot&\notag\int_{[U_{2r-3}]}\LL( \int_{[U_{2r-2}]} \int_{[U_{2r-1}]}
    \phi(u_{2r-1}u_{2r-2}u_{2r-3}g) \psi_{N_{2r}}(u_{2r-1})\, du_{2r-1}\, du_{2r-2}\RR)\\
    \cdot&\LL( \int_{[U_{2r-2}]} F(u_{2r-2} u_{2r-3}g, s_1,s_2)\psi_{N_{2r}}(u_{2r-2})\, du_{2r-2} \RR)\psi_{N_{2r}}(u_{2r-3})\, du_{2r-3}.\notag
\end{align}
Similarly, from Corollary \ref{cor: independence} we have
\begin{align*}
    \int_{[U_{2r-2}]} F(u_{2r-2} u_{2r-3}g,s_1,s_2)&\psi_{N_{2r}}(u_{2r-2})\, du_{2r-2}\\
    =& \int_{[U_{2r-3}]} \int_{[U_{2r-2}]}
    F(u_{2r-2}u_{2r-3}g,s_1,s_2)
    \psi_{N_{2r}}(u_{2r-2})\, du_{2r-2}\,
    du_{2r-3},
\end{align*}
and substituting this identity into \eqref{eqn: I-r integral III}, we get
\begin{align*}
    I_r(\varphi_{\pi_{2r}},\phi,s_1,s_2) =
    &\int_{Z_{2r}(\A) U_{2r-1}(\A) U_{2r-2}(\A) U_{2r-3}(\A) U_{2r-4}(k) \cdots U_1(k) \sm \GL_{2r}(\A)}
    W_{\varphi_{\pi_{2r}}}^\psi(g)\\
    \cdot&\int_{[U_{2r-3}]} \int_{[U_{2r-2}]} \int_{[U_{2r-1}]}
    \phi(u_{2r-1}u_{2r-2}u_{2r-3}g)
    \psi_{N_{2r}}(u_{2r-1}u_{2r-3})\, du_{2r-1}\,du_{2r-2}\,du_{2r-3}\\
    \cdot&\int_{[U_{2r-3}]} \int_{[U_{2r-2}]}
    F(u_{2r-2}u_{2r-3}g, s_1,s_2) 
    \psi_{N_{2r}}(u_{2r-2})\, du_{2r-2}\, du_{2r-3}.
\end{align*}
Then, proceeding in the same manner, we obtain
\begin{align*}
    I_r(&\varphi_{\pi_{2r}},\phi,s_1,s_2) =
    \int_{Z_{2r}(\A) N_{2r}(\A) \sm \GL_{2r}(\A)}
    W_{\varphi_{\pi_{2r}}}^\psi(g)   
    W_{\phi, 2r}(g) \\
    \cdot&\int_{[U_1]} \cdots \int_{[U_{2r-3}]} \int_{[U_{2r-2}]}
    F(u_{2r-2}u_{2r-3}\cdots u_1g,s_1,s_2)
    \psi_{N_{2r}}(u_{2r-2}u_{2r-4}\cdots u_2)\, du_{2r-2}\, du_{2r-3}\cdots du_1,
\end{align*}
where $W_{\phi, 2r}(g)$ is given by
\begin{align*}
    W_{\phi, 2r}(g) = \int_{[N_{2r}]} \phi(n g) \psi_{N_{2r}}^{\operatorname{odd}}(n) \, dn.
\end{align*}
From the definition of $N_{2r-1} \subset \GL_{2r-1}$ and $\psi_{N_{2r}}^{\operatorname{even}}$ we have
\begin{align*}
    I_r(\varphi_{\pi_{2r}},\phi,s_1,s_2) =
    \int_{Z_{2r}(\A) N_{2r}(\A) \sm \GL_{2r}(\A)}
    &W_{\varphi_{\pi_{2r}}}^\psi(g)
    W_{\phi, 2r}(g)\\
    \cdot&\int_{N_{2r-1}(k) \sm N_{2r-1}(\A)}
    F(\begin{pmatrix}
        n &\\ & 1
    \end{pmatrix}g, s_1,s_2) \psi_{N_{2r}}^{\operatorname{even}}(\begin{pmatrix}
        n & \\ & 1
    \end{pmatrix})\, dn.
\end{align*}
Thus, to prove the identities in Proposition \ref{prop: unfolding}, it suffices to show that 
\begin{align}
\label{eqn: unfolding f-W(2r-1)}
    &\int_{[N_{2r-1}]}
    F(\begin{pmatrix}
        n &\\ & 1
    \end{pmatrix}g, s_1,s_2) \psi_{N_{2r}}^{\operatorname{even}}(\begin{pmatrix}
        n & \\ & 1
    \end{pmatrix})\, dn 
    = f_{W_{2r-1}}(g, s_1,s_2),\\
\label{eqn: unfold E(r,r)}
    &\int_{[N_{2r}]}
    E(ng, s) \psi_{N_{2r}}^{\operatorname{odd}}(n)\, dn  = W_{(r,r);2r}(g,s),
\end{align}
where $f_{W_{2r-1}}(g, s_1,s_2)$ and $W_{(r,r);2r}(g,s)$ are as defined in \eqref{f-W(2r-1)} and \eqref{eqn: W(r,r)(s)} respectively.
We will only detail the proof of the identity \eqref{eqn: unfolding f-W(2r-1)}, as one can perform similar arguments to derive \eqref{eqn: unfold E(r,r)}. 
To show \eqref{eqn: unfolding f-W(2r-1)}, we first observe that $P_{r,r-1,1}\sm P_{2r-1,1} \cong P_{r,r-1}\sm \GL_{2r-1}$ such that 
\begin{align*}
    F(g, s_1,s_2) = 
    \sum_{\gamma \in P_{r,r-1}(k)\sm \GL_{2r-1}(k)}
    f(\begin{pmatrix}
        \gamma & \\ & 1
    \end{pmatrix}g, s_1, s_2).
\end{align*}
Moreover, by the Bruhat decomposition we have
\begin{align*}
    P_{r,r-1}(k)\sm \GL_{2r-1}(k)/ N_{2r-1}(k)\cong W_{P_{r,r-1}}\sm W_{\GL_{2r-1}} \cong \{A\subset\{1,\dots, 2r-1\}: |A| = r\},
\end{align*}
where $W_{P_{r,r-1}}$ and $W_{\GL_{2r-1}}$ are the Weyl group of $P_{r,r-1}$ and $\GL_{2r-1}$ respectively. Thus, 
\begin{align*}
    \GL_{2r-1}(k)= \bigsqcup_{\substack{
    A\subset\{1,\dots,2r-1\}\\
    |A| = r
    }} P_{r,r-1}(k) w_A N_{2r-1}(k),
\end{align*}
where $w_A$ is the permutation given as follows. If $A = \{a_1 < a_2 < \dots < a_r\} \subset \{1,\dots, 2r-1\}$ and $A^c = \{1,\dots, 2r-1\}\sm A = \{c_1 < c_2 < \dots < c_{r-1}\}$ then $w_A$ is the permutation satisfying
\begin{align*}
    w_A(a_i) = i, && w_A(c_s) = r+s, 
\end{align*}
for $1 \leq i \leq r, 1 \leq s \leq r-1$. Therefore,
\begin{align*}
    \int_{[N_{2r-1}]} F(\begin{pmatrix}
        n&\\ & 1
    \end{pmatrix}g,& s_1,s_2) \psi_{N_{2r}}^{\operatorname{even}}(\begin{pmatrix}
        n&\\ & 1
    \end{pmatrix})\, dn \\
    =&
    \sum_{\substack{
     A \subset \{1,\dots, 2r-1\}\\
    |A| = r
    }} \int_{N_{w_A}(k) \sm N_{2r-1}(\A)}
    f(w_A \begin{pmatrix}
        n & \\ & 1
    \end{pmatrix}g, s_1, s_2) 
    \psi_{N_{2r}}^{\operatorname{even}}(\begin{pmatrix}
        n&\\ & 1
    \end{pmatrix})\, dn,
\end{align*}
where $N_{w_A} = N_{2r-1} \cap w_A^{-1} P_{r,r-1} w_A$ is given by
\begin{align}
    \label{eqn: N-wA}
    N_{w_A} = \{ 
    n = (n_{i,j}) \in N_{2r-1} \mid 
    \text{$n_{i,j} =0$ whenever $i \in A^c, j \in A, i < j$}
    \}.
\end{align}
In particular, for $w_0$ and $U_{w_0}$ defined in the proposition, we have $w_0 = w_{A_0}$, $N_{2r-1} = N_{w_{A_0}} U_{w_0}$ and $N_{w_{A_0}} \cap U_{w_0} = \{1\}$ for $A_0 = \{1,3,5,\dots, 2r-1\}$. Furthermore, the character $\psi_{N_{2r}}^{\operatorname{even}}$ is trivial on $N_{w_{A_0}}(\A)$, realised as a subgroup of $N_{2r}(\A)$. Thus, to prove \eqref{eqn: unfolding f-W(2r-1)} it suffices to show that 
\begin{align}
\label{eqn: integral N-wA}
    \int_{N_{w_A}(k) \sm N_{w_A}(\A)} f(w_A \begin{pmatrix}
        n & \\ & 1
    \end{pmatrix}h, s_1, s_2) \psi_{N_{2r}}^{\operatorname{even}}(\begin{pmatrix}
        n & \\ & 1
    \end{pmatrix})\, dn
\end{align}
evaluates to zero for all such $w_A \neq w_0$ and $h \in \GL_{2r}(\A)$.
To do so, we will show for $w_A \neq w_0$, there exists a root subgroup
\begin{align*}
    U_{p,p+1} = \{I_{2r-1} + a E_{p,p+1}\} 
\end{align*}
for some $p \in \{2,4,\dots, 2r-2\}$ and some normal subgroup $N_{w_A}'\subset N_{w_A}$ such that $N_{w_A} = U_{p,p+1} N_{w_A}'$ and $U_{p,p+1} \cap N_{w_A}' = \{1\}$. Indeed, for $w_A \neq w_0$, one can choose 
\begin{align*}
    p_A = 2\min\{t \in \{1,2,\dots, r-1\} : \text{$w_A(2t) \leq r$ or $w_A(2t+1) >r$}\} \in \{2,4,\dots, 2r-2\},
\end{align*}
and defining $l_{p_A} : N_{w_A} \to \G_a$ by
\begin{align*}
    l_{p_A}(n) = n_{p_A, p_A+1},
\end{align*}
we have
\begin{align*}
    N_{w_A}' = \ker(l_{p_A}) = \{n = (n_{i,j})_{1\leq i, j \leq 2r-1} \in N_{w_A} \mid n_{p_A,p_A + 1} =0\},
\end{align*}
such that $N_{w_A} = U_{p_A,p_A+1}N_{w_A}'$ and $U_{p_A, p_A + 1} \cap N_{w_A}' = \{1\}$. Consequently, we obtain \eqref{eqn: unfolding f-W(2r-1)} and hence, this concludes the proof of Proposition \ref{prop: unfolding}.

By the uniqueness of the Whittaker model, we have the following corollary.
\begin{cor}
    \label{cor: Euler factorisation}
    Assuming $\operatorname{Re}(s_1-s),\operatorname{Re}(2rs_1-s_2),\operatorname{Re}(s),\operatorname{Re}(s_1), \operatorname{Re}(s_2) \gg0$ and for factorisable integration data, the global integrals $I_r(\varphi_{\pi_{2r}}, \theta_{\tau_2}, s_1,s_2)$ and $I_r(\varphi_{\pi_{2r}}, s, s_1,s_2)$ factor as
    \begin{align*}
        I_r(\varphi_{\pi_{2r}}, \theta_{\tau_2}, s_1,s_2) =& \prod_\nu \calZ_r(W_\nu, s_1,s_2),\\
        I_r(\varphi_{\pi_{2r}}, s, s_1, s_2) =& \prod_\nu \calZ_r(W_\nu,s,s_1,s_2),
    \end{align*}
    where
    \begin{align*}
        \calZ_r(W_{\nu}, s_1,s_2) =&\int_{Z_{2r}(k_\nu) N_{2r}(k_\nu) \sm \GL_{2r}(k_\nu)}
        W_{\varphi_{\pi_{2r,\nu}}}^{\psi_\nu}(g)
        W_{\theta_{\tau_{2,\nu}}, 2r}(g) 
        f_{W_{2r-1,\nu}}(g, s_1,s_2) \, dg,\\
        \calZ_r(W_{\nu},s,s_1,s_2) =& \int_{Z_{2r}(k_\nu) N_{2r}(k_\nu) \sm \GL_{2r}(k_\nu)}
        W_{\pi_{{2r},\nu}}^{\psi}(g) 
        W_{(r,r);2r,\nu}(g,s)
        f_{W_{{2r-1},\nu}}(g,s_1,s_2)\,dg,
    \end{align*}
    and
    \begin{align*}
        f_{W_{2r-1},\nu}(g,s_1,s_2)=& \int_{U_{w_0}(k_\nu)}
        f_\nu(w_0 ng, s_1,s_2) \psi_{N_{2r},\nu}^{\operatorname{even}}(n) \, dn,\\
        W_{(r,r);2r,\nu}(g,s) =& \int_{U_{w_0'}(k_\nu)} f_\nu(w_0' u g, s) \psi_{U_{w_0'},\nu}(u)\, du.
    \end{align*}
\end{cor}
\section{Unramified computation}
\label{sec: Unramified computation}
Throughout this section, let $F$ be a non-archimedean local field of characteristic zero, with ring of integers $\OO_F$ and $\varpi \in \OO_F$ a uniformizer such that $|\varpi| = q^{-1}$ for $|\cdot| = |\cdot|_F$ being the usual absolute value on $F$, and $q$ is the cardinality of the residue field. We will also assume the measures are normalised such that
\begin{align*}
    \operatorname{vol}(\OO_F, dx) = 1, &&
    \operatorname{vol}(\OO_F^\tm , d^\tm x) = 1, &&
    \operatorname{vol}(\GL_n(\OO_F), dk)= 1.
\end{align*}
Next, we let
\begin{align}
\label{eqn: pi-2r, tau-2 irreducible unramified PS}
   \pi_{2r} = \Ind_{B_{2r}(F)}^{\GL_{2r}(F)} (\xi_1 \otimes \xi_2 \otimes \cdots \otimes \xi_{2r}), &&  \tau_2= \Ind_{B_2(F)}^{\GL_2(F)} (\chi \otimes \chi^{-1}),
\end{align}
be irreducible unramified principal series of $\PGL_{2r}(F)$ and $\PGL_2(F)$ where $\xi_i$ and $\chi$ are unramified characters of $F^\tm$ satisfying $\prod_{i=1}^{2r} \xi_i = \mathbf 1$. We set their corresponding Satake parameters to be 
\begin{align}
    \label{eqn: Satake parameters}
    t_{\pi_{2r}} = \diag(\alpha_1,\dots, \alpha_{2r}) \in \SL_{2r}(\C), && 
    t_{\tau_2} = \diag(b,b^{-1}) \in \SL_2(\C).
\end{align}
We will fix $\psi : F \to \C^\tm$ an unramified additive character of $F$, i.e. it is trivial on $\OO_F$ but non-trivial on $\varpi^{-1} \OO_F$. We will define the unramified analogue of $\psi_{N_{2r}}$ (resp. $\psi_{N_{2r}}^{\operatorname{odd}}, \psi_{N_{2r}}^{\operatorname{even}}$ and $\psi_{U_{w_0'}}$) accordingly. Let $W_{\pi_{2r}}^\circ \in \mathcal W(\pi_{2r}, \psi_{N_{2r}})^{\GL_{2r}(\OO_F)}$ and $W_{\Delta(\tau_2,r)}^\circ \in \mathcal W(\Delta(\tau_{2}, r), (\psi_{N_{2r}}^{\operatorname{odd}})^{-1})^{\GL_{2r}(\OO_F)}$, $f_{s_1,s_2}^\circ = f^\circ(\cdot,s_1,s_2) \in I(s_1,s_2)^{\GL_{2r}(\OO_F)}$ and $f_s^\circ = f^\circ(\cdot,s) \in I(s)^{\GL_{2r}(\OO_F)}$ be the normalised spherical functions normalised by
\begin{align*}
   W_{\pi_{2r}}^\circ(I_{2r}) = W_{\Delta(\tau_2,r)}^\circ(I_{2r}) = f^\circ_{s_1,s_2}(I_{2r}) = f_s^\circ(I_{2r}) = 1.
\end{align*}
With these as well as the notations introduced in  Proposition \ref{prop: unfolding}, we define
\begin{align}
\label{eqn: f-W(2r-1)-unramified}
    f_{W_{2r-1}^\circ}(g,s_1,s_2) =& \int_{U_{w_0}(F)} f_{s_1,s_2}^\circ(w_0 n g) \psi_{N_{2r}}^{\operatorname{even}}(n)\, dn,\\
\label{eqn: W-(r,r);2r-unramified}
    W^\circ_{(r,r);2r}(g,s) =& \int_{U_{w_0'}(F)} f_s^\circ(w_0' u g, s) \psi_{U_{w_0'}}(u)\, du.
\end{align}
Furthermore, we will define the following local $L$-functions:
\begin{align*}
    \zeta(s) =& (1-q^{-s})^{-1},&&
    L(\pi_{2r}, s) = \prod_{i=1}^{2r}
    (1-\alpha_i q^{-s})^{-1},\\
    L(\pi_{2r}, \wedge^2, s) =& \prod_{1 \leq i < j \leq 2r} (1-\alpha_i\alpha_j q^{-s})^{-1},&&
    L(\pi_{2r} \tm \tau_2,s) = \prod_{i=1}^{2r}
    (1-\alpha_ibq^{-s})^{-1}
    (1-\alpha_ib^{-1}q^{-s})^{-1}.
\end{align*}
We will evaluate unramified local integrals $\calZ_r^\circ(W^\circ,s_1,s_2)$ and $\calZ_r^\circ(W^\circ,s,s_1,s_2)$ defined by
\begin{align*}
    \calZ_r^\circ(W^\circ,s_1,s_2) =& \int_{Z_{2r}(F) N_{2r}(F) \sm \GL_{2r}(F)}
    W_{\pi_{2r}}^\circ(g) W_{\Delta(\tau_2,r)}^\circ(g)
    f_{W_{2r-1}^\circ}(g,s_1,s_2)\,dg,\\
    \calZ_r^\circ(W^\circ,s,s_1,s_2) =& \int_{Z_{2r}(F) N_{2r}(F) \sm \GL_{2r}(F)}
        W_{\pi_{{2r}}}^{\circ}(g) 
        W_{(r,r);2r}^\circ(g,s)
        f_{W_{2r-1}^\circ}(g,s_1,s_2)\,dg.
\end{align*}
We have the following theorem.
\begin{thm}
    \label{thm: unramified computation}
    For $\operatorname{Re}(s_1-s),\operatorname{Re}(2rs_1-s_2),\operatorname{Re}(s),\operatorname{Re}(s_1), \operatorname{Re}(s_2) \gg0$,
    the unramified local integrals $\calZ_r^\circ(W^\circ,s_1,s_2)$ and $\calZ_r^\circ(W^\circ,s,s_1,s_2)$ defined above evaluate to 
    \begin{align*}
         \calZ_r^\circ(W^\circ, s_1,s_2)  =& \frac{
        L\LL(\pi_{2r} \tm \tau_2, rs_1 - \frac{r-1}{2} \RR) L\LL(\pi_{2r},\wedge^2,s_2\RR) 
        }{
        \displaystyle\zeta(rs_2) 
        \zeta(2rs_1+(r-1)s_2-r+1)
        \prod_{i=1}^{r-1} \zeta(2rs_1 - s_2 -i + 1)
        },\\
        \calZ_r^\circ(W^\circ,s,s_1,s_2) =& \frac{
        L\LL(\pi_{2r}, r(s_1+s-1) + \frac{1}{2} \RR)
        L\LL(\pi_{2r}, r(s_1-s) + \frac{1}{2} \RR)
        L\LL(\pi_{2r},\wedge^2,s_2\RR) 
        }{
        \displaystyle \zeta(rs_2) 
        \zeta(2rs_1+(r-1)s_2-r+1)
        \prod_{i=1}^{r-1} \zeta(2rs_1 - s_2 -i + 1)
        \prod_{j=1}^r \zeta(2rs - r + j)
        }.
    \end{align*}
\end{thm}
Similar to the previous section, we will introduce some preliminary identities for the terms $W_{\Delta(\tau_2, r)}^\circ$, $W_{(r,r);2r}^\circ$ and $f_{W_{2r-1}^\circ}$ before we proceed with the proof of Theorem \ref{thm: unramified computation}.
\subsubsection{Rank-one computation}
Let $\mu_i$ be unramified characters of $F^\tm$ be such that the normalised induction $\Ind_{B_2(F)}^{\GL_2(F)}(\mu_1 \otimes \mu_2)$ is an irreducible unramified principal series of $\GL_2(F)$. Let $\Phi^\circ \in (\Ind_{B_2(F)}^{\GL_2(F)}(\mu_1 \otimes \mu_2))^{\GL_2(\OO_F)}$ be its normalised spherical section such that $\Phi^\circ(k) = \Phi^\circ(I_2) = 1$ for all $k \in \GL_2(\OO_F)$. We define the $\GL_2(F)$-Jacquet integral as
\begin{align}
    \label{eqn: Jacquet integral GL(2)}
    J_{\mu_1,\mu_2}(g) = \int_F \Phi^\circ(\begin{pmatrix}
        &1\\
        1&
    \end{pmatrix} \begin{pmatrix}
        1& x\\ & 1
    \end{pmatrix} g) \psi(x) \, dx, && g \in \GL_2(F).
\end{align}
\begin{lem}
    \label{lem: Rank-one Jacquet integral}
    (Rank-one Jacquet integral) In the usual convergence cone, we have
    \begin{align*}
        J_{\mu_1,\mu_2}(I_2) = 1- q^{-1}\mu_1\mu_2^{-1}(\varpi) .
    \end{align*}
    Moreover, for $m\in \Z$ we have
    \begin{align*}
        \frac{J_{\mu_1,\mu_2}(\diag(\varpi^m,1))}{J_{\mu_1,\mu_2}(I_2)}
         = \begin{dcases*}
            q^{-m/2} \sum_{l=0}^{m}\mu_1(\varpi)^{l} \mu_2(\varpi)^{m-l}, & if $m\geq0$,\\
            0 &otherwise.
        \end{dcases*}
    \end{align*}
\end{lem}
\begin{proof}
    See  \cite[Theorem 4.6.5]{Bu}.
\end{proof}
\subsection{Reduction to rank-one computations}
For $a_1,a_2,\dots, a_{2r-1} \in F^\tm$, write
\begin{align}
    \label{eqn: t} t = \diag(a_1a_2\cdots a_{2r-1},
        a_2 \cdots a_{2r-1},\dots, a_{2r-1}, 1) \in \GL_{2r}(F).
\end{align}
Furthermore assuming $\operatorname{Re}(s) \gg0$, we set $\mu: F^\tm \to \C^\tm$ to be the unramified character of $F^\tm$ defined by $\mu(a) = |a|^{rs-r/2}$, and let $\Xi_2$ be the corresponding irreducible unramified principal series of $\GL_2(F)$ given by
\begin{align}
\label{eqn: Xi2}
    \Xi_2 = \Ind_{B_2(F)}^{\GL_2(F)}(\mu \otimes \mu^{-1}),
\end{align}
with $\mathcal W(\Xi_2, \psi_{N_2}^{-1})$ being its corresponding Whittaker model. Similarly, we write $\mathcal W(\tau_2, \psi_{N_2}^{-1})$ to denote the Whittaker model of $\tau_2$, with $W_{\tau_2}^\circ \in \mathcal W(\tau_2, \psi_{N_2}^{-1})^{\GL_2(\OO_F)}$ (resp. $W_{\Xi_2}^\circ \in \mathcal W(\Xi_2, \psi_{N_2}^{-1})^{\GL_2(\OO_F)}$) being its normalised unramified Whittaker function such that $W_{\tau_2}^\circ(I_2) = W_{\Xi_2}^\circ(I_2) = 1$.
For such $t \in \GL_{2r}(F)$ given in \eqref{eqn: t}, we have the following product formula for $W_{\Delta(\tau_2, r)}^\circ(t)$ (resp. $W_{(r,r);2r}^\circ(t,s)$) in terms of $W_{\tau_2}^\circ$ (resp. $W_{\Xi_2}^\circ$) given as follows.
\begin{prop}
    \label{prop: theta-tau-(r,r) closed formula}
    For $\operatorname{Re}(s_1-s),\operatorname{Re}(2rs_1-s_2),\operatorname{Re}(s),\operatorname{Re}(s_1), \operatorname{Re}(s_2) \gg0$ and $t \in \GL_{2r}(F)$ given in \eqref{eqn: t}, we have
    \begin{align*}
        W_{\Delta(\tau_2,r)}^\circ(t) =&\begin{dcases*}
            \D_{P_{2^r}}^{1/4}(t) \cdot \prod_{i=1}^r W_{\tau_2}^\circ(\diag(a_{2i-1}, 1)) &if $|a_1|, |a_3|, \dots, |a_{2r-1}|\leq1$,\\
            0, &otherwise,
        \end{dcases*}\\
         W_{(r,r);2r}^\circ(t,s) =& \prod_{i=1}^r \frac{1}{\zeta(2rs-r+i)}\begin{dcases*}
             \D_{P_{2^r}}^{1/4}(t) \cdot \prod_{i=1}^r W_{\Xi_2}^\circ(\diag(a_{2i-1}, 1))&if $|a_1|, |a_3|, \dots, |a_{2r-1}|\leq1$,\\
            0, &otherwise,
         \end{dcases*}
    \end{align*}
    Furthermore, $f_{W_{2r-1}^\circ}(t,s_1,s_2) =0$ unless $|a_2|, |a_4|, \dots, |a_{2r-2}|\leq1$. If this condition holds, we have
    \begin{align*}
        f_{W_{2r-1}^\circ}(t,s_1,s_2) =&\LL(
        \prod_{i=1}^r |a_{2i-1}|^{rs_1+(i-1)s_2+(i-1)(r-i)}
        \RR) \LL( \prod_{i=1}^{r-1} |a_{2i}|^{is_2+i(r-i)} \RR)\\
        \cdot& \frac{
        \zeta(2rs_1-s_2-(r-1))^{r-1}
        }{
        \prod_{i=0}^{r-2}
        \zeta(2rs_1-s_2-i)
        }
        \prod_{i=1}^{r-1}
        (1-q^{r-1-(2rs_1-s_2)} |a_{2i}|^{(2rs_1-s_2)-(r-1)}).
    \end{align*}
\end{prop}
\begin{proof}
    We will first give a proof of the identity for  $W_{\Delta(\tau_2,r)}^\circ$, followed by $W_{(r,r);2r}^\circ$ and lastly $f_{W_{2r-1}^\circ}$. 
    Throughout, we will assume that $\operatorname{Re}(s_1-s),\operatorname{Re}(2rs_1-s_2),\operatorname{Re}(s),\operatorname{Re}(s_1), \operatorname{Re}(s_2) \gg0$.
    Recalling $\tau_2$ in \eqref{eqn: pi-2r, tau-2 irreducible unramified PS}, and from \cite[Claim 9]{CFGK} we have
    \begin{align*}
        \Delta(\tau_2, r) \cong \Ind_{P_{r,r}(F)}^{\GL_{2r}(F)}(\chi \circ \det \otimes \chi^{-1}\circ \det),
    \end{align*}
    where elements $f \in  \Delta(\tau_2, r)$ can be regarded as smooth complex-valued functions defined on $\GL_{2r}(F)$ satisfying the quasi-invariance
    \begin{align*}
        f(\begin{pmatrix}
            g_1 & X\\ & g_2 
        \end{pmatrix}h) = \chi(\det g_1) \chi^{-1}(\det g_2) \LL| \frac{\det g_1}{\det g_2} \RR|^{r/2} f(h),
    \end{align*}
    for $g_1,g_2 \in \GL_r(F)$, $X \in M_{r,r}(F)$ and $h \in \GL_{2r}(F)$. Let $f_{\Delta(\tau_2,r)}^\circ \in \Delta(\tau_2, r)^{\GL_{2r}(\OO_F)}$ be its normalised spherical element such that $f_{\Delta(\tau_2,r)}^\circ(I_{2r}) = f_{\Delta(\tau_2,r)}^\circ(k) = 1$ for all $k \in \GL_{2r}(\OO_F)$.
    We define the functional $J_r \in \operatorname{Hom}_{N_{2r}(F)}(\Delta(\tau_2, r), \overline{\psi_{N_{2r}}^{\operatorname{odd}}})$ by
    \begin{align*}
        J_r(f) = \int_{N_{w_0'}(F) \sm N_{2r}(F)}
        f(w_0' n) \psi_{N_{2r}}^{\operatorname{odd}}(n) \, dn,
    \end{align*}
    where $w_0' \in \GL_{2r}(F)$ is the permutation matrix given in Proposition \ref{prop: unfolding},
    and $N_{w_0'} := N_{2r} \cap w_0'^{-1} P_{r,r} w_0'$ is given by
    \begin{align*}
        N_{w_0'} = \LL\{ n = (n_{i,j}) \in N_{2r} \mid n_{2a-1,2b} = 0, \forall 1 \leq a \leq b \leq r \RR\}.
    \end{align*}
    Since $N_{2r} = N_{w_0'} U_{w_0'}$ and $N_{w_0'} \cap U_{w_0'} = \{1\}$, then using the notations in Proposition \ref{prop: unfolding}, we can re-write $J_r(f)$ defined above as
\begin{align}
    \label{eqn: Jacquet integral II}
    J_r(f) = \int_{(UT)_r(F)} f(w_0' u(X)) \psi(\tr X)\,dX,
\end{align}
Thus, we have an integral representation of $W_{\Delta(\tau_2,r)}^\circ$ given by
\begin{align}
\label{eqn: integral repn W-theta-2r}
    W_{\Delta(\tau_2, r)}^\circ(g) = \frac{
    J_r(\Delta(\tau_2, r)(g) f_{\Delta(\tau_2,r)}^\circ)
    }{
    J_r(f_{\Delta(\tau_2,r)}^\circ)
    }, && g \in \GL_{2r}(F).
\end{align}
We will use this integral expression for $W_{\Delta(\tau_2, r)}^\circ$ to derive the identity in Proposition \ref{prop: theta-tau-(r,r) closed formula}. Define $w_L \in \GL_{2r}$ by
\begin{align*}
    w_L = \diag(\begin{pmatrix}
        &1\\ 
        1&
    \end{pmatrix}, \begin{pmatrix}
        &1\\ 
        1&
    \end{pmatrix},\dots, \begin{pmatrix}
        &1\\ 
        1&
    \end{pmatrix}),
\end{align*}
and set $\sigma = w_0 \in \GL_{2r}$ to be the permutation matrix defined in Proposition \ref{prop: unfolding}. It is clear that  $w_0' = \sigma w_L$. Let $X =(x_{i,j}) \in (UT)_r(F)$ be such that $X = D + Z$ where $D = \diag(x_{1,1},\dots, x_{r,r})$ and $Z = (x_{i,j})_{i<j}$ is strictly upper triangular. We define
    \begin{align*}
        n(D) = \prod_{i=1}^r (I_{2r} + x_{i,i} E_{2i-1,2i}), && v_0(Z) = \prod_{1 \leq i < j \leq r} (I_{2r} + x_{i,j} E_{2i-1,2j}), && u'(Z) = \prod_{1 \leq i < j \leq r}(I_{2r} + x_{i,j}E_{2i,2j-1}),
    \end{align*}
    so that $u(X) \in U_{w_0'}$ takes the form $u(X) = n(D) v_0(Z)$ and $w_0' u(X) = \sigma u'(Z) w_L n(D)$.
    Consequently, the Jacquet integral $J_r(f)$ in \eqref{eqn: Jacquet integral II} can be written as
    \begin{align}
        \label{eqn: Jacquet integral III}
            J_r(f) = \int_{F^r} (A_\sigma f)(w_L n(\diag(x_{1,1}, x_{2,2},\dots, x_{r,r}))) \psi(x_{1,1}+x_{2,2}+\cdots+x_{r,r})\,dx_{i,i},
    \end{align}
    for
    \begin{align*}
        (A_\sigma f)(g) = \int_{U_\sigma(F)} f(\sigma ug) \, du,
    \end{align*}
    where $U_\sigma$ is the unipotent subgroup of $N_{2r}$ consisting of elements of the form $u'(Z)$.
    Let $P_{2^r} = L_{2^r} U_{2^r}$ be the standard parabolic subgroup of $\GL_{2r}(F)$ associated to the composition $(2^r)$ of $2r$, where its Levi subgroup $L_{2^r}$ consists of matrices of the form $\diag(g_1,\dots, g_r)$ for $g_i \in \GL_2(F)$. Then, for $u'(Z) \in U_\sigma$ and
    \begin{align*}
        d(u,v) = \diag(u_1,v_1,u_2,v_2,\dots, u_r,v_r) \in L_{2^r},
    \end{align*}
    we have
    \begin{align*}
        d(u,v)^{-1} u'(Z) d(u,v) = u'(Z'), && Z' = (z_{i,j}')_{i<j},&& z_{i,j}' = z_{i,j} \frac{u_j}{v_i}.
    \end{align*}
    Then, using the quasi-invariance of $f_{\Delta(\tau_2, r)}^\circ$ and by changing suitable variables, we have
    \begin{align*}
        (A_\sigma f_{\Delta(\tau_2,r)}^\circ)(d(u,v) g) = \prod_{i=1}^r \chi(u_i) \chi^{-1}(v_i) |u_i|^{\rho_i + 1/2} |v_i|^{\rho_i-1/2}(A_\sigma f_{\Delta(\tau_2,r)}^\circ)(g), 
    \end{align*}
    where $\rho_i = (r+1-2i)/2$ for $1 \leq i \leq r$. Moreover, it can be verified that
    \begin{align*}
        (A_\sigma f_{\Delta(\tau_2,r)}^\circ)(\diag(\begin{pmatrix}
            1 & x_1\\ & 1
        \end{pmatrix},\begin{pmatrix}
            1 & x_2\\ & 1
        \end{pmatrix},\dots, \begin{pmatrix}
            1 & x_r\\ & 1
        \end{pmatrix}) g) = (A_\sigma f_{\Delta(\tau_2,r)}^\circ)(g).
    \end{align*}
    Thus by the Iwasawa decomposition for $\GL_2(F)$, we have
    \begin{align}
    \label{eqn: A-sigma auxiliary I}
        (A_\sigma f_{\Delta(\tau_2,r)}^\circ)(\diag(g_1,\dots, g_r)) = (A_\sigma f_{\Delta(\tau_2,r)}^\circ)(I_{2r})\cdot \prod_{i=1}^r |\det g_i|^{\rho_i} \phi_{\tau_2}^\circ(g_i),
    \end{align}
    where $\phi_{\tau_2}^\circ \in \tau_2^{\GL_2(\OO_F)}$ is the normalised spherical function such that $\phi_{\tau_2}^\circ(I_2)= 1$. Furthermore, since
    \begin{align*}
        \D_{P_{2^r}}^{1/4}(\diag(g_1,\dots, g_r)) =\prod_{i=1}^r |\det g_i|^{\rho_i}, 
    \end{align*}
    then the above expression for $A_\sigma f_{\Delta(\tau_2,r)}^\circ$ given in \eqref{eqn: A-sigma auxiliary I} becomes
    \begin{align}
        \label{eqn: A-sigma auxiliary II}
        (A_\sigma f_{\Delta(\tau_2,r)}^\circ)(\diag(g_1,\dots, g_r)) = (A_\sigma f_{\Delta(\tau_2,r)}^\circ)(I_{2r})\cdot \D_{P_{2^r}}^{1/4}(\diag(g_1,\dots, g_r)) \prod_{i=1}^r \phi_{\tau_2}^\circ(g_i).
    \end{align}
    Substituting the above identity \eqref{eqn: A-sigma auxiliary II} for $A_\sigma f_{\Delta(\tau_2, r)}^\circ$ into \eqref{eqn: Jacquet integral III}, we have
    \begin{align*}
        J_r(\Delta(\tau_2,r)(m) f_{\Delta(\tau_2,r)}^\circ) = (A_\sigma f_{\Delta(\tau_2,r)}^\circ)(I_{2r})\cdot \D_{P_{2^r}}^{1/4}(m)
        \prod_{i=1}^r \LL(\int_F \phi_{\tau_2}^\circ(\begin{pmatrix}
            &1\\
            1&
        \end{pmatrix} \begin{pmatrix}
            1 & x_{i,i}\\
            &1
        \end{pmatrix} g_i) \psi(x_{i,i})\,dx_{i,i} \RR),
    \end{align*}
    for $m= \diag(g_1,\dots,g_r) \in \GL_{2r}(F)$ where $g_i \in \GL_2(F)$. Furthermore, from \eqref{eqn: integral repn W-theta-2r} and Lemma \ref{lem: Rank-one Jacquet integral}, together with the non-vanishing of $(A_\sigma f_{\Delta(\tau_2,r)}^\circ)(I_{2r})$ due to the Gindikin-Karpelevich formula \cite[Theorem 3.1]{Ca}, we obtain
    \begin{align}
    \label{eqn: W-theta-2r auxiliary II}
        W_{\Delta(\tau_2,r)}^\circ(m) = \D_{P_{2^r}}^{1/4}(m) \prod_{i=1}^r W_{\tau_2}^\circ(g_i).
    \end{align}
    Choosing $m = t$ where $t \in \GL_{2r}(F)$ as in \eqref{eqn: t}, we obtain the expression for $W_{\Delta(\tau_2, r)}^\circ(t)$ in the proposition. On the other hand, for $m = \diag(g_1,\dots,g_r) \in \GL_{2r}(F)$ where $g_i \in \GL_2(F)$, one can repeat the same argument as above to obtain the following identity for $W_{(r,r);2r}^\circ(m,s)$:
    \begin{align*}
        W_{(r,r);2r}^\circ(m,s) = W_{(r,r);2r}^\circ(I_{2r},s)\cdot \D_{P_{2^r}}^{1/4}(m)
                             \prod_{i=1}^r W_{\Xi_2}^\circ(g_i).
    \end{align*}
    Here, $W_{(r,r);2r}^\circ(I_{2r},s)$ can be written as
    \begin{align*}
        W_{(r,r);2r}^\circ(I_{2r},s) = (A_\sigma f_s^\circ)(I_{2r})\cdot \LL( \int_F\Phi^\circ(\begin{pmatrix}
            & 1\\ 1& 
        \end{pmatrix}\begin{pmatrix}
            1 & x\\ & 1
        \end{pmatrix})  \psi(x) \, dx \RR)^r,
    \end{align*}
    where $\Phi^\circ \in \Xi_2^{\GL_2(\OO_F)}$ is the normalised spherical function of $\Xi_2$ given in \eqref{eqn: Xi2} such that $\Phi^\circ(k)= \Phi^\circ(I_2) = 1$ for all $k \in \GL_2(\OO_F)$. 
    Using Lemma \ref{lem: Rank-one Jacquet integral} together with the Gindikin-Karpelevich formula \cite[Theorem 3.1]{Ca} for $(A_\sigma f_s^\circ)(I_{2r})$, we obtain the following closed formula for $W_{(r,r);2r}^\circ(I_{2r},s)$:
    \begin{align*}
        W_{(r,r);2r}^\circ(I_{2r},s) = \prod_{i=1}^r \frac{1}{\zeta(2rs-r+i)}.
    \end{align*}
    Again, choosing $m = t$ where $t \in \GL_{2r}(F)$ as in \eqref{eqn: t}, we obtain the expression for $W_{(r,r);2r}^\circ(t,s)$ in the proposition. Finally, for $f_{W_{2r-1}^\circ}(t,s_1,s_2)$ we follow a similar argument as above. First, we observe that $f_{W_{2r-1}^\circ}(g,s_1,s_2)$ given in \eqref{eqn: f-W(2r-1)-unramified} can be reduced to 
    \begin{align}
        \label{eqn: f-W(2r-1) auxiliary I}
        f_{W_{2r-1}^\circ}(g,s_1,s_2) = \int_{(UT)_{r-1}(F)} f_{s_1,s_2}^\circ(w_0 \widetilde n(l) g) \psi_{N_{2r}}(\widetilde n(l)) \, dl,
    \end{align}
    where for $l = (l_{i,j})_{1\leq i, j \leq r-1} \in (UT)_{r-1}(F)$ we define $\widetilde n(l) \in \GL_{2r}(F)$ to be
    \begin{align*}
        \widetilde n(l) = I_{2r} + \sum_{1 \leq i \leq j \leq r-1} l_{i,j} E_{2i,2j+1}.
    \end{align*}
    For such $l = (l_{i,j})_{1\leq i, j \leq r-1} \in (UT)_{r-1}(F)$, we split $l = \widetilde D + \widetilde Z$ as
    \begin{align*}
        \widetilde D = \diag(l_{1,1}, l_{2,2},\dots, l_{r-1,r-1}), && \widetilde Z = (l_{i,j})_{1\leq i < j \leq r-1},
    \end{align*}
    with
    \begin{align*}
         n_1(\widetilde D) = \prod_{i=1}^{r-1} (I_{2r} + l_{i,i} E_{2i,2i+1}), && 
         \widetilde n_0(\widetilde Z) = I_{2r} + \sum_{1 \leq i < j \leq r-1} l_{i,j} E_{2i,2j+1},
    \end{align*}
    such that $\widetilde n(l) = n_1(\widetilde D) \tilde n_0(\widetilde Z)$. Also, writing
    \begin{align*}
        \widetilde w_L = \diag(1, \begin{pmatrix}
            &1\\
            1&
        \end{pmatrix},\begin{pmatrix}
            &1\\
            1&
        \end{pmatrix},\dots, \begin{pmatrix}
            &1\\
            1&
        \end{pmatrix}, 1) \in \GL_{2r}(F),
    \end{align*}
    and setting $\widetilde \sigma$ such that $w_0 = \widetilde \sigma \widetilde w_L$, we have $ w_0 \widetilde n(l) = \widetilde \sigma u_1(\widetilde Z) \widetilde w_L n_1(\widetilde D)$ for
    \begin{align*}
         u_1(\widetilde Z) = I_{2r} + \sum_{1 \leq i < j \leq r-1} l_{i,j} E_{2i+1,2j}.
    \end{align*}
    Let  $U_{\widetilde \sigma}$ be the unipotent subgroup consisting of matrices of the form $u_1(\widetilde Z)$, the integral $f_{W_{2r-1}^\circ}$ defined in \eqref{eqn: f-W(2r-1) auxiliary I} can be written as 
    \begin{align}
        \label{eqn: f-W(2r-1) auxiliary II}
        f_{W_{2r-1}^\circ}(g, s_1, s_2) =\int_{F^{r-1}} (A_{\widetilde \sigma} f_{s_1,s_2}^\circ)(\widetilde w_L
        n_1(\widetilde D) g) \psi(l_{1,1} + \cdots + l_{r-1,r-1}) \, dl_{i,i},
    \end{align}
    where $A_{\widetilde \sigma} f_{s_1,s_2}^\circ$ is given by
    \begin{align*}
        (A_{\widetilde \sigma} f_{s_1,s_2}^\circ)(g) = \int_{U_{\widetilde \sigma}(F)} f_{s_1,s_2}^\circ(\widetilde \sigma u g) \, du, && g \in \GL_{2r}(F).
    \end{align*}
    Let $Q = P_{1,2^{r-1},1}$ be the standard parabolic subgroup of $\GL_{2r}(F)$ corresponding to the composition $(1,2^{r-1}, 1)$ whose Levi is $L_Q \cong \GL_1 \tm \GL_2^{r-1} \tm \GL_1$. Let $t_0, u_i, v_i, t_{r} \in \GL_1$ for $1 \leq i \leq r-1$, we define $d \in L_Q \subset \GL_{2r}(F)$ by
    \begin{align*}
        d = \diag(t_0, u_1, v_1, u_2, v_2,\dots, u_{r-1}, v_{r-1}, t_{r}).
    \end{align*}
    We observe that
    \begin{align*}
        \widetilde\sigma d \widetilde\sigma^{-1} = \diag(\diag(t_0, u_1,\dots, u_{r-1}), \diag(v_1,\dots, v_{r-1}), t_r),
    \end{align*}
    and
    \begin{align*}
        d^{-1} u_1(\widetilde Z) d = u_1(\tilde Z'), && \widetilde Z' = (\widetilde z_{i,j}')_{1\leq i < j \leq r-1}, &&
        \widetilde z_{i,j}' = \widetilde z_{i,j} \frac{u_j}{v_i}.
    \end{align*}
    Therefore, using the quasi-invariance of $f_{s_1,s_2}^\circ$ in \eqref{eqn: f(s1,s2) quasi-invariance} and performing suitable changes of variables, we have
    \begin{align}
        \label{eqn: A-tilde-sigma}
        (A_{\widetilde \sigma} f^\circ_{s_1,s_2})(dg)  = |t_0|^{rs_1} |t_{r}|^{-rs_1-(r-1)s_2} \prod_{i=1}^{r-1}
        |u_i|^{rs_1-i+1} |v_i|^{-rs_1 + s_2+(r-1)-i} (A_{\widetilde \sigma} f^\circ_{s_1,s_2})(g).
    \end{align}
    Moreover for $g \in \GL_{2r}(F)$, a direct calculation shows that
    \begin{align*}
        (A_{\widetilde \sigma} f_{s_1,s_2}^\circ)(\diag(1, \begin{pmatrix}
            1 & x_1\\ & 1
        \end{pmatrix}, \begin{pmatrix}
            1 & x_2\\ & 1
        \end{pmatrix}, \dots, \begin{pmatrix}
            1 & x_{r-1}\\ & 1   
        \end{pmatrix},1) g) = (A_{\widetilde \sigma} f_{s_1,s_2}^\circ)(g).
    \end{align*}
    Hence, from the Iwasawa decomposition for $\GL_2(F)$, we can re-write $A_{\widetilde \sigma} f_{s_1,s_2}^\circ$ as follows
    \begin{align}
        \label{eqn: A-tilde-sigma III}
        (A_{\widetilde \sigma} f_{s_1,s_2}^\circ)(m) =  (A_{\tilde \sigma} f_{s_1,s_2}^\circ)(I_{2r})\cdot
        |t_0|^{rs_1} |t_{r}|^{-rs_1-(r-1)s_2} \prod_{i=1}^{r-1} \phi_i^\circ(g_i),
    \end{align}
    for $m = \diag(t_0, g_1, g_2,\dots, g_{r-1}, t_r) \in \GL_{2r}(F)$ and $t_0, t_{r} \in F^\tm, g_i \in \GL_2(F)$, where $\phi_i^\circ$ are the smooth complex-valued function on $\GL_2(F)$ satisfying the following quasi-invariance:
    \begin{align*}
        \phi_i^\circ(\begin{pmatrix}
            t_1 & x\\ & t_2
        \end{pmatrix} g k) = |t_1|^{rs_1-i+1} |t_2|^{-rs_1 + s_2 +(r-1) -i} \phi_i^\circ(g),
    \end{align*}
    where $t_i \in F^\tm, x \in F, g \in \GL_2(F)$ and $k \in \GL_2(\OO_F)$ such that $\phi_i^\circ(I_2) = 1$.  Moreover,  $(A_{\widetilde \sigma} f_{s_1,s_2}^\circ)(I_{2r})$ has the closed formula
    \begin{align*}
        (A_{\widetilde \sigma} f_{s_1,s_2}^\circ)(I_{2r}) = \prod_{i=1}^{r-2} \LL( \frac{
        \zeta(2rs_1-s_2-i)
        }{\zeta(2rs_1-s_2-i+1)} \RR)^i,
    \end{align*}
    by the Gindikin-Karpelevich formula \cite[Theorem 3.1]{Ca}. Thus, for $t \in \GL_{2r}(F)$ given in \eqref{eqn: t} and substituting \eqref{eqn: A-tilde-sigma III} into \eqref{eqn: f-W(2r-1) auxiliary II}, together with Lemma \ref{lem: Rank-one Jacquet integral}, we obtain the identity for $f_{W_{2r-1}^\circ}(t,s_1,s_2)$ in the proposition. This concludes the proof.
\end{proof}
\subsection{Cauchy-Littlewood type identity}
Next, we will prove a Cauchy-Littlewood type identity.
Let $(k_1,k_2,\dots, k_n) \in \Z_{\geq0}^n$ be an $n$-tuple such that $k_1\geq k_2 \geq \cdots \geq k_n \geq0$. Given $t \in \SL_{n}(\C)$, we let $\chi^{\SL_n(\C)}_{(k_1,k_2,\dots, k_n)}(t)$ be the character of the irreducible representation of $\SL_{n}(\C)$ corresponding to the highest weight $(k_1,k_2,\dots, k_n)$, evaluated at $t$. Throughout, we will set
\begin{align*}
    x = q^{-rs_1 + (r-1)/2}, && y = q^{-s_2},
\end{align*}
and $\alpha = \diag(\alpha_1,\dots, \alpha_{2r}) \in \SL_{2r}(\C)$ and $b \in \C^\tm$. For $n = (n_1,\dots, n_{2r-1}) \in \Z_{\geq0}^{2r-1}$, we also set
\begin{align}
    \label{eqn: lambda-n}
    \lambda(n) = (n_1 + \cdots + n_{2r-1}, n_2 +\cdots + n_{2r-1}, \dots, n_{2r-1}, 0).
\end{align}
We have the following lemma.
\begin{lem}
    \label{lem: Cauchy-Littlewood} 
For $\operatorname{Re}(2rs_1-s_2), \operatorname{Re}(s_1), \operatorname{Re}(s_2)\gg0$, we have
\begin{align*}
    \sum_{n_1,\dots, n_{2r-1}\geq0} \chi_{\lambda(n)}^{\SL_{2r}(\C)} (\alpha)
    \prod_{j=1}^{r}\chi_{(n_{2j-1},0)}^{\SL_2(\C)}
    (\diag(b,b^{-1})) 
    x^{\sum_{j=1}^r n_{2j-1}}
    y^{\sum_{j=1}^{r-1} j(n_{2j} + n_{2j+1})} \prod_{j=1}^{r-1}
    \frac{
    (1-(x^2y^{-1})^{n_{2j}+1}
    }{
    1-x^2y^{-1}
    }
\end{align*}
evaluating to 
\begin{align*}
    \frac{
    (1-y^r)(1-x^2 y^{r-1})
    }{
    \displaystyle \prod_{i=1}^{2r}(1-xb\alpha_i)(1-xb^{-1}\alpha_i)
    \prod_{1 \leq i < j \leq 2r}(1-y \alpha_i \alpha_j)
    }.
\end{align*}
\end{lem}
\begin{proof}
We write $s_\lambda$ (resp. $s_{\lambda/\mu}$) to denote a Schur (resp. skew-Schur) polynomial. Recalling the standard Littlewood identity \cite[Chapter I, \textsection 5, Example 5(b)]{Mac}:
\begin{align*}
    \prod_{i < j} (1 -y \alpha_i \alpha_j)^{-1}
    = \sum_{\nu_1 \geq \cdots \nu_r \geq0}
    y^{\nu_1 + \cdots + \nu_r} s_{(\nu_1,\nu_1,\dots, \nu_r, \nu_r)}(\alpha),
\end{align*}
and since $\alpha \in \SL_{2r}(\C)$, we have a simplified expression
\begin{align}
\label{eqn: Littlewood identity II}
    (1-y^r)\prod_{i<j}(1-y\alpha_i\alpha_j)^{-1}
    = 
    \sum_{\mu_1\geq \cdots \geq\mu_{r-1}\geq0}
    y^{\sum_i \mu_i} s_{\eta(\mu)}(\alpha),
\end{align}
where
\begin{align*}
    \eta(\mu) = (\mu_1,\mu_1,\mu_2,\mu_2,\dots, \mu_{r-1}, \mu_{r-1}, 0,0).
\end{align*}
Setting $u = xb, v = xb^{-1}$ so that $uv = x^2$, the skew-Cauchy identity \cite[Chapter I, \textsection 5, Example 26(1)]{Mac} gives us
\begin{align}
\label{eqn: skew-schur}
    s_{\eta(\mu)}(\alpha) \prod_{i=1}^{2r}\LL( (1-u\alpha_i) (1-v\alpha_i) \RR)^{-1} = \sum_{\lambda} s_{\lambda/\eta(\mu)}(u,v) s_\lambda(\alpha),
\end{align}
where the sum on the right ranges over all partitions of length at most $2r$. In particular, we obtain
\begin{align*}
    \prod_{i=1}^{2r}\LL( (1-u\alpha_i) (1-v\alpha_i) \RR)^{-1}\cdot \sum_{\mu_1\geq \cdots \geq\mu_{r-1}\geq0} y^{\sum_i \mu_i} s_{\eta(\mu)}(\alpha)  =
    \sum_{\mu_1\geq \cdots \geq\mu_{r-1}\geq0} 
    y^{\sum_i \mu_i}
     \sum_{\lambda} s_{\lambda/\eta(\mu)}(u,v) s_\lambda(\alpha).
\end{align*}
Substituting the identity \eqref{eqn: skew-schur} into the equation above, we have
\begin{align}
\label{eqn: skew-schur II}
    \frac{
    1-y^r
    }{
    \prod_{i=1}^{2r}(1-u\alpha_i)(1-v\alpha_i) \prod_{1 \leq i < j \leq 2r}(1-y\alpha_i \alpha_j)
    } = 
    \sum_{\substack{\lambda\\
    \mu_1\geq \cdots \geq \mu_{r-1}\geq0}} 
    y^{\sum_i \mu_i}
    s_{\lambda/\eta(\mu)}(u,v) s_\lambda(\alpha).
\end{align}
Suppose the following summand 
\begin{align}
\label{eqn: summand}
    y^{\sum_i \mu_i}
    s_{\lambda/\eta(\mu)}(u,v) s_\lambda(\alpha)
\end{align}
is nonzero for a partition $\lambda = (\lambda_1,\dots, \lambda_{2r})$ and $\mu_1\geq \cdots \geq \mu_{r-1} \geq0$. Recall that a skew Schur polynomial in two variables vanishes if its skew diagram has a column of height greater than two. Thus, one has
\begin{align}
    \label{eqn: skew-schur constraint}
    \lambda_{2i} \geq \mu_i \geq \lambda_{2i+1}, && \forall 1 \leq i \leq r-1,
\end{align}
and in particular $\mu_{r-1} \geq \lambda_{2r-1} \geq \lambda_{2r}$, so $\mu_i \geq \lambda_{2r}$ for $1 \leq i \leq r-1$. Moreover, defining 
\begin{align*}
    \lambda' = \lambda - \lambda_{2r}(1^{2r}), && 
    \mu' = \mu - \lambda_{2r}(1^{r-1}),
\end{align*}
then again using the fact that $\alpha \in \SL_{2r}(\C)$ we have
\begin{align*}
    s_\lambda(\alpha) = s_{\lambda'}(\alpha).
\end{align*}
Moreover, since the skew diagram contains $\lambda_{2r}$ forced columns of height $2$, each contributing $uv$, we also have
\begin{align*}
    s_{\lambda/\eta(\mu)}(u,v) = (uv)^{\lambda_{2r}} s_{\lambda'/\eta(\mu')}(u,v),
\end{align*}
where $\sum_{i=1}^{r-1} \mu_i = \sum_{i=1}^{r-1} \mu_i' + (r-1) \lambda_{2r}$. Thus, the term on the right of \eqref{eqn: skew-schur II} becomes
\begin{align}
\label{eqn: skew-schur aux}
    \sum_{\lambda_{2r}\geq0}(uvy^{r-1})^{\lambda_{2r}} \cdot \sum_{\substack{\lambda = (\lambda_1\geq \lambda_2 \geq \cdots \geq \lambda_{2r-1} \geq0)\\ 
    \mu_1 \geq \mu_2 \cdots \geq \mu_{r-1}\geq0}}
    y^{\sum_i \mu_i}
    s_{\lambda/\eta(\mu)}(u,v)
    s_\lambda(\alpha).
\end{align}
Substituting \eqref{eqn: skew-schur aux} into \eqref{eqn: skew-schur II}, we obtain
\begin{align}
    \label{eqn: skew-schur III}
    \frac{
    (1-y^r)(1-uvy^{r-1})
    }{
    \prod_i (1-u\alpha_i)(1-v\alpha_i)
    \prod_{i<j} (1-y \alpha_i \alpha_j)
    } = \sum_{\substack{\lambda = (\lambda_1\geq \lambda_2 \geq \cdots \geq \lambda_{2r-1} \geq0)\\ 
    \mu_1 \geq \mu_2 \cdots \geq \mu_{r-1}\geq0}}
    y^{\sum_i \mu_i}
    s_{\lambda/\eta(\mu)}(u,v)
    s_\lambda(\alpha).
\end{align}
Furthermore, setting $n_i = \lambda_i - \lambda_{i+1}$ then under the constraint \eqref{eqn: skew-schur constraint}, we have the following closed formula for $s_{\lambda/\eta(\mu)}(u,v)$:
\begin{align}
    \label{eqn: skew Schur polynomial}
    s_{\lambda/\eta(\mu)}(u,v) = (uv)^{\sum_{i=1}^{r-1}(\lambda_{2i} - \mu_i)}
    \prod_{j=1}^r s_{(n_{2j-1},0)}(u,v).
\end{align}
Substituting the above closed formula \eqref{eqn: skew Schur polynomial} together with the following
\begin{align*}
    \sum_{\mu_i =\lambda_{2i+1}}^{\lambda_{2i}} 
    y^{\mu_i}(uv)^{\lambda_{2i} - \mu_i}
    = y^{\lambda_{2i}}
    \frac{
    1-(uvy^{-1})^{n_{2i}+1}
    }{
    1-uvy^{-1}
    },&&
    \sum_{i=1}^{r-1}\lambda_{2i} =
    \sum_{i=1}^{r-1}
    i(n_{2i} + n_{2i+1}),
\end{align*}
into \eqref{eqn: skew-schur III}, we obtain our desired identity. This concludes the proof.
\end{proof}
With these preliminary results, we can proceed with the proof of Theorem \ref{thm: unramified computation}. We first compute the local unramified integral $\calZ_r^\circ(W^\circ,s_1,s_2)$. By Iwasawa decomposition, we have
\begin{align*}
        \calZ_r^\circ(W^\circ,s_1,s_2) = \int_{Z_{2r}(F) \sm T_{2r}(F)} 
        \D_{B_{2r}}^{-1}(t)
        W_{\pi_{2r}}^\circ(t) 
        W_{\Delta(\tau_2,r)}^\circ(t)
        f_{W_{2r-1}}^\circ(t,s_1,s_2)\,dt,
    \end{align*}
We parametrise elements of $Z_{2r}(F) \sm T_{2r}(F)$ by $t \in \GL_{2r}(F)$ in \eqref{eqn: t} and set $n_i = \nu_F(a_i)$. By the Casselman-Shalika-Shintani \cite{CS,Sh} formula, we have
\begin{align*}
    W_{\pi_{2r}}^\circ(t) =& \begin{dcases*}
        \D_{B_{2r}}^{1/2}(t) \chi_{\lambda(n)}^{\SL_{2r}(\C)} (t_{\pi_{2r}}), &if $n_i \geq0$,\\
        0 &otherwise,
    \end{dcases*} && 
    W_{\tau_2}^\circ(\diag(\varpi^m, 1)) = \begin{dcases*}
        q^{-m/2} \chi^{\SL_2(\C)}_{(m,0)}(t_{\tau_2}), &if $m\geq0$,\\
        0 &otherwise,
    \end{dcases*}
\end{align*}
where  $t_{\pi_{2r}}\in \SL_{2r}(\C)$ and $t_{\tau_2} \in \SL_2(\C)$ are the Satake parameters given in \eqref{eqn: Satake parameters} and $\lambda(n)$ is as in \eqref{eqn: lambda-n}.
With these, together with the closed formulas for $W_{\Delta(\tau_2,r)}^\circ(t)$ and $f_{W_{2r-1}^\circ}(t,s_1,s_2)$ given in Proposition \ref{prop: theta-tau-(r,r) closed formula}, we have
\begin{align*}
    &\prod_{j=0}^{r-2}\zeta(2rs_1-s_2-j) \cdot \calZ_r^\circ(W^\circ,s_1,s_2)\\
    &=\sum_{n_1,\dots, n_{2r-1}\geq0}
    \chi_{\lambda(n)}^{\SL_{2r}(\C)} (t_{\pi_{2r}})
    \LL(\prod_{i=1}^r \chi_{(n_{2i-1},0)}^{\SL_2(\C)}(t_{\tau_2}) \RR)
    x^{\sum_{j=1}^r n_{2j-1}}
    y^{\sum_{j=1}^{r-1} j(n_{2j} + n_{2j+1})} \prod_{j=1}^{r-1}
    \frac{
    (1-(x^2y^{-1})^{n_{2j+1}}
    }{
    1-x^2y^{-1}
    },
\end{align*}
for $x = q^{-rs_1+(r-1)/2}$ and $y = q^{-s_2}$. From Lemma \ref{lem: Cauchy-Littlewood}, we obtain the desired identity for $\calZ_r^\circ(W^\circ,s_1,s_2)$ in Theorem \ref{thm: unramified computation}. As for the other integral $\calZ_r^\circ(W^\circ,s,s_1,s_2)$, we follow the same argument to obtain a similar expression for $\calZ_r^\circ(W^\circ,s,s_1,s_2)$:
\begin{align*}
    \calZ_r^\circ(W^\circ,s,s_1,s_2) = \int_{Z_{2r}(F) \sm T_{2r}(F)} \D_{B_{2r}}^{-1}(t)
    W_{\pi_{2r}}^\circ(t)
    W^\circ_{(r,r);2r}(t,s) 
    f_{W_{2r-1}^\circ}(t,s_1,s_2)\,dt.
\end{align*}
Setting $b = q^{rs-r/2} \in \C^\tm$ in Lemma \ref{lem: Cauchy-Littlewood} and using the closed formulas for $W_{(r,r);2r}^\circ(t,s)$ and $f_{W_{2r-1}^\circ}(t,s_1,s_2)$ given in Proposition \ref{prop: theta-tau-(r,r) closed formula}, we obtain the desired identity for $\calZ_r^\circ(W^\circ,s,s_1,s_2)$. This concludes the proof of Theorem \ref{thm: unramified computation}.

\end{document}